\documentclass[11pt]{amsart}
\usepackage{latexsym}
\usepackage{amsmath,amsthm,amsfonts,amscd,eucal,mathabx,amssymb,mathrsfs,pictex}
\usepackage{graphicx}
\usepackage{hyperref}

\numberwithin{equation}{section}

\def\ca{{\mathcal A}}
\def\cb{{\mathcal B}}
\def\cc{{\mathcal C}}
\def\cd{{\mathcal D}}
\def\ce{{\mathcal E}}

\def\ch{{\mathcal H}}
\def\ci{{\mathcal I}}

\def\ck{{\mathcal K}}
\def\cl{{\mathcal L}}

\def\car{{\mathcal R}}
\def\cs{{\mathcal S}}
\def\ct{{\mathcal T}}

\def\ga{{\mathfrak A}}

\def\gs{{\mathfrak S}}

\def\ba{{\mathbb A}}

\def\bc{{\mathbb C}}

\def\bn{{\mathbb N}}

\def\bq{{\mathbb Q}}
\def\br{{\mathbb R}}
\def\bt{{\mathbb T}}

\def\bz{{\mathbb Z}}

\def\a{\alpha}
\def\b{\beta}
\def\g{\gamma}  \def\G{\Gamma}
\def\d{\delta}  \def\D{\Delta}

\def\eps{\varepsilon}

\def\l{\lambda} \def\L{\Lambda}

\def\m{\mu}

\def\n{\nu}
\def\r{\rho}
\def\s{\sigma} 
\def\t{\tau}
\def\f{\varphi}  
\def\th{\theta} 
\def\om{\omega} 
\def\z{\zeta}

\def\id{\hbox{id}}
\def\ker{\hbox{Ker}}

\newtheorem{thm}{Theorem}[section]

\newtheorem{prop}[thm]{Proposition}

\theoremstyle{definition}
\newtheorem{rem}[thm]{Remark}
\newtheorem{defin}[thm]{Definition}

\def\!{\mskip-\thinmuskip}

\def\sign{\mathop{\rm sign}}
\def\supp{\mathop{\rm supp}}

\newcommand{\ty}[1]{\mathop{\rm {#1}}}
\def\di{{\rm d}}
\def\id{{\rm id}}

\def\idd{{1}\!\!{\rm I}}

\newcommand{\snorm}{[\!]}

\DeclareMathAlphabet{\mathpzc}{OT1}{pzc}{m}{it}

\begin{document}

\title[spectral triples and Fredholm modules]
{Modular spectral triples and deformed Fredholm modules}
\author{Fabio Ciolli}
\address{Fabio Ciolli\\
Dipartimento di Matematica \\
Universit\`{a} di Roma Tor Vergata\\
Via della Ricerca Scientifica 1, Roma 00133, Italy} \email{{\tt
ciolli@mat.uniroma2.it}}
\author{Francesco Fidaleo}
\address{Francesco Fidaleo\\
Dipartimento di Matematica \\
Universit\`{a} di Roma Tor Vergata\\
Via della Ricerca Scientifica 1, Roma 00133, Italy} \email{{\tt
fidaleo@mat.uniroma2.it}}
\date{\today}


\begin{abstract}

In the setting of non-type $\ty{II_1}$ representations, we propose a definition of {\it deformed Fredholm module} $\big[D_\ct|D_\ct|^{-1}\,,\,{\bf\cdot}\,\big]_\ct$ for a modular spectral triple $\ct$, where $D_\ct$ is the deformed Dirac operator. $D_\ct$ is assumed to be invertible for the sake of simplicity, and its domain is an ``essential" operator system $\ce_\ct$. According to such a definition, we obtain $\big[D_\ct|D_\ct|^{-1}\,,\,{\bf\cdot}\,\big]_\ct=|D_\ct|^{-1}d_\ct(\,{\bf\cdot}\,)+d_\ct(\,{\bf\cdot}\,)|D_\ct|^{-1}$, where $d_\ct$ is the deformed derivation associated to $D_\ct$.
Since the ``quantum differential" $1/|D_\ct|$ appears in a symmetric position, such a definition of 
Fredholm module differs from the usual one even in the undeformed case, that is in the tracial case.
Therefore, it seems to be more suitable for the investigation of noncommutative manifolds in which the nontrivial modular structure might play a crucial role. We show that all models in \cite{FS} of non-type $\ty{II_1}$ representations of noncommutative 2-tori indeed provide 
modular spectral triples, and in addition deformed Fredholm modules according to the definition proposed in the present paper. Since the detailed knowledge of the spectrum of the Dirac operator plays a fundamental role in spectral geometry, we provide a characterisation of eigenvalues and eigenvectors of the deformed Dirac operator $D_\ct$ in terms of the periodic solutions of a particular class of eigenvalue Hill equations. 
\end{abstract}
\maketitle

\tableofcontents

\section{Introduction}
It is well known that in noncommutative approach to mathematics, the concept of ``point" is meaningless. The same happens in quantum physics in an attempt to unify electromagnetic, weak and strong interactions with the gravitation where, at Planck scale, any reasonable measurement process is lost. 

Another fundamental long standing problem in number theory, and thus in a different 
direction even if connected with the previous problem, concerns the attempt to prove the Riemann hypothesis about the location of the zeroes of the Riemann zeta-function $\z(z)$, 
see e.g.\  \cite{C1}.
Therefore, motivated by the two potential applications described above, Connes' \emph{noncommutative geometry} grew impetuously in the last decades.

Regarding other relevant applications of noncommutative geometry, we mention that the {\it quantum Hall effect} can be explained by the Connes-Chern character ${\rm Ch}(P_F)$ of the eigen-projection of the Hamiltonian on energies smaller than or equal to the Fermi level. Therefore, for some relevant applications, the pairing between the cyclic cohomology and $K$-theory plays a crucial role.

We note that it would be difficult to provide an exhaustive list of the literature. However, for a natural starting point
we refer the reader to the seminal monographs \cite{C0, CPR, GBVF}, and the standard available references on this topic. 

Taking into account the above introductory facts, a ``noncommutative manifold" might describe the physics of the quantum Hall effect (see e.g.\  \cite{BSvE} and the references therein), or that arising from Connes' standard model and the relative {\it spectral action }(cf.\ \cite{CL, CC}). 
Also, any commutative manifold can be considered as a particular case of a noncommutative one. In this last case, Connes' reconstruction theorem (cf.\ \cite{C0}) allows to reconstuct all relevant properties of the manifold under consideration in terms of the natural axioms of noncommutative geometry in which the concept of spectral triple assumes a quite relevant role.

Therefore, one of the main ingredient in noncommutative geometry is then the so-called (undeformed) \emph{spectral triple} and the associated {\it Fredholm module}, the last being also  the candidate to encode many important properties of the manifold under consideration, and then to keep together K-theory, topology, measure theory and many others. 

On the other hand, the discovery of a deep connections between the quantum statistical mechanics,  in particular the {\it Kubo-Martin-Schwinger} (KMS for short) {\it condition} (cf.\ \cite{Ku, MS, HHW}) and, on the mathematical side, the Tomita theory (e.g.\ \cite{T0}) dealing with fundamental questions of operator algebras (see e.g.\ \cite{C00}),  suggests the necessary need to take the modular data into account.

To the knowledge of the authors, the idea to take into consideration the modular data in defining the spectral triples has been started in \cite{CM0} with the perspective of the application to the theory of foliations. 

The project of exhibiting twisted spectral triples according to the modular data
is indeed carried out in detail in \cite{CM}. For the
spectral triples considered in \cite{CM}, the twisting comes from inner and bounded perturbations of the canonical trace of the noncommutative 2-torus, thus in the context of type $\ty{II_1}$ representations. Such spectral triples were called {\it modular}. 

Many other works have been devoted to the investigation of deformed spectral triples whose twist arises, or does not directly arise, from the modular data. We mention e.g.\ 
\cite{CPPR, GMT, MO} and, for some potential applications to quantum field theory, \cite{LM}. 

The program to provide examples of twisted spectral triples 
arising from type $\ty{III}$ representations, and thus exhibiting a highly nontrivial modular structure, is carried out in \cite{FS} for  noncommutative 2-tori.
In this paper, it was considered irrational rotations by the angle $2\pi\a$, provided that $\a$ is a Liouville number, that is an irrational number admitting a ``fast" approximation by rationals. 

Since the previous mentioned paper, as well as the present one, can be considered as a kind of generalisation of \cite{CM}, we denote the spectral triples considered here again as modular, with the aim to point out the fact that the (highly nontrivial) modular structure plays a crucial role.
 
Such a preliminary but essential step to construct non-type $\ty{II_1}$ representations of the noncommutative torus is done by considering suitable 
$C^\infty$-diffeomorphisms $\mathpzc{f}$ of the unit circle $\bt$. Then 
the \emph{Gel'fand-Naimark-Segal} (GNS for short) representations of the $C^*$-algebra $\ba_{\a}$ of the noncommutative torus generated by states $\om_\mathpzc{f}\in\cs(\ba_\a)$ with central support in the bidual, and canonically associated to the diffeomorphism $\mathpzc{f}$,  provide the searched type $\ty{III}$ representations.

The modular spectral triples constructed in \cite{FS}, whose associated twisted commutator is automatically defined to be ``real" in its own domain, satisfy the following basic requirements. The first one concerns the associated twist of the Dirac operator and the relative twisted derivation, which should come from the Tomita  modular operator. We remark  that, in the non type $\ty{II_1}$ cases, the modular operator is unbounded and not inner, in particular it should not come from a bounded inner perturbation. 

According to the standard definition of any spectral triple, such a suitably deformed Dirac operator should have compact resolvent, and in addition the unbounded twisted derivation should include in its domain sufficiently many elements (e.g.\ typically a dense $*$-subalgebra or, as considered in the recent paper \cite{CS}, merely an essential {\it operator system}). We note that the examples of modular spectral triples constructed in \cite{CM} for the noncommutative 2-torus arise from bounded (indeed ``smooth") inner perturbations of the canonical trace, and thus provide only type $\ty{II_1}$ GNS representations.

The previous facts are relevant because they constitute the minimal requirements for a spectral triple to encode the main metric properties and the associated measure theoretic ones, that is the noncommutative counterpart of the volume form, i.e.\ a state.
Moreover, the spectral triple may also furnish the possible pairing with the K-Theory and the equivariant K-Theory, and thus the corresponding local index formula for which the novel ingredient associated to the non triviality of the modular data may play a crucial role.

Concerning the last aspect involving pairings with K-theory and K-ho\-mo\-lo\-gy, the concept of {\it Fredholm module} plays a crucial role, see e.g.\ \cite{C0, Had}, and the literature cited therein. Therefore, it is natural to address the investigation about Fredholm modules arising from such new spectral triples, suitably deformed by the use of the modular data. 

By the considerations set out above, it is natural to expect that such Fredholm modules should be again deformed according to the underlying modular structure. 
Hence, with $D_\ct$ and $\ce_\ct$ the 
deformed Dirac operator and the operator system in the domain of the deformed commutator 
associated to the modular spectral triple $\ct$, we provide a definition of a deformed commutator $[D_\ct|D_\ct|^{-1}, A]_\ct$ between the ``sign" $F_\ct=D_\ct|D_\ct|^{-1}$ of the Dirac operator and the elements $A$ of $\ce_\ct$. According to such a definition, this deformed commutator differs from the usual one also in the tracial case when the Tomita modular operator is trivial.

Under some technical assumptions about the domains of the involved unbounded operators, for the deformed commutator in the definition of the Fredholm module, and for $d_\ct$ the deformed derivation associated to $\ct$, 
we show that 
\begin{equation}
\label{basic}
[D_\ct|D_\ct|^{-1}, A]_\ct=|D_\ct|^{-1}d_\ct(A)+ d_\ct(A)|D_\ct|^{-1}\,,
\end{equation}
and thus it still produces compact operators.

We want to point out that the ``quantum differential" $1/|D_\ct|$ appears in a symmetric manner in the above formula. In addition, in order to obtain such a symmetric displacement of the quantum differential, the Fredholm modules associated to undeformed spectral triples arising from the tracial cases should also deformed in a canonical way.

It would then be expected that such a new definition of 
Fredholm module may be more suitable for the investigation of some relevant noncommutative manifolds in which the nontrivial modular structure might play a crucial role.
Similarly, such objects arising from modular spectral triples and deformed Fredholm modules (e.g.\ the cohomologies), should take into account the non-cyclicity associated to the presence of the modular operator. Therefore, they might be new objects which have to be identified yet. For general details concerning previous investigations involving these aspects of cyclic cohomology, the reader is referred to the monograph \cite{C0}, and the survey paper \cite{K}.
\vskip.3cm

The present paper is organised as follows. Apart from the introduction and a preliminary section in which some basic notions are reported, in Section \ref{mstripa} we provide a slightly modified definition (w.r.t.\ Definition 2.4 in \cite{FS}) of modular spectral triple which, according to \cite{CS}, takes also into account the possibility to deal with operator systems.

Section \ref{modstr} is devoted to show that the models arising from non-type $\ty{II_1}$ representations of the noncommutative 2-torus described in Sections 7 and 8 of \cite{FS}, indeed provide nontrivial examples of modular spectral triples according to Definition \ref{mstgl}. 

Section \ref{spqu} deals with the new definition of Fredholm module for which, under natural assumptions on domains of the involved unbounded operators, we prove the key-formula \eqref{basic}. We also prove that all above mentioned examples for the noncommutative torus enjoy the properties of providing deformed Fredholm modules. 

Section \ref{ex} concerns with a simple but illustrative one dimensional example. Compared with the twin one in Section 3.1 of \cite{FMR}, such an example helps the reader to clarify the framework. 

In Section \ref{oeost}, it is outlined that the combination of examples in Section 9 of \cite{FS} and those in Section 7 of \cite{FH}, indeed providing deformed Fredholm modules. We are leaving to the reader the involved technical computations because those do not add anything else to the conceptual meaning of the topic. 
For some of such examples of deformed Fredholm modules, their domains contain a dense $*$-algebra, but the price to pay is that the intrinsic undeformed Dirac operator must be suitably modified by using the growth sequence of the diffeomorphism $\mathpzc{f}$, hence according to the modular state $\om_\mathpzc{f}$, depending on $\mathpzc{f}$ and
entering in the definition of spectral triple. Therefore, the arising Fredholm modules are losing some degree of intrinsicity.  

Since the detailed knowledge of the spectrum of the Dirac operator plays a fundamental role, we added Section \ref{hkil} in which the crucial problem to diagonalise such deformed Dirac operators is equivalent to find all periodic solutions of a class of eigenvalue equations, firstly introduced by G. W. Hill in \cite{Hi}. 
\vskip.3cm

We end by mentioning some further investigations (i.e.\ metric aspects, spectral geometry, index theory, cohomologies and many others) involving modular spectral triples and deformed Fredholm modules considered in the present paper. The reader is referred to \cite{C0} and the survey-paper \cite{CPR} for the corresponding objects arising from the tracial (i.e.\ untwisted) standard situation. Therefore, it should be desirable 
to come back to these aspects somewhere else in future.

\section{Preliminaries}
\label{prrrr}

We gather here some facts useful for the forthcoming sections.
\smallskip

\noindent
\textbf{Basic notions.}
We denote by $\bt$ the unit circle 
$$
\{z\in\bc\mid \text{s.t.}\, |z|=1\} \sim\{\th\in[0,2\pi)\mid 0\,\,\text{is identified with}\,2\pi\}\,,
$$
where $z=e^{\imath\th}$.
It is a compact group whose Haar measure is the normalised Lebesgue measure
$$
\mathpzc{m}:=\frac{\di\th}{2\pi}=\frac{\di z}{2\pi\imath z}\,,\quad z=e^{\imath \th}\in\bt\,.
$$

If $Y$ is a point-set, $\cb(Y)$ denotes the unital $C^*$-algebra consisting of all bounded complex-valued functions on $Y$.
\smallskip

Given an irrational number which is not diophantine, it is possible to prove that it is a Liouville number, that is it satisfies the condition:
\begin{itemize}
\item[{\bf (L)}] A {\it Liouville number} $\a\in(0,1)$ is a real number such that for each $N\in\bn$ the inequality 
\begin{equation*}
\left|\a - \frac{p}{q}\right|< \frac{1}{q^N}
\end{equation*}
has an infinite number of solutions for $p,q\in\bn$ with $(p,q)=1$.
\item[{\bf (UL)}] Among Liouville numbers, we also consider those such that, for each $\l>1$ and $N\in\bn$, the inequality
\begin{equation*}
\left|\a - \frac{p}{q}\right|<\frac{1}{\l^{q^N}}
\end{equation*}
again admits infinite number of solutions for $p,q$ as above.
\end{itemize}
The numbers satisfying {\bf (UL)}, which have a prescribed faster approximation by rationals, are called {\it ultra-Liouville numbers}.
Concerning the Liouville numbers, the reader is referred to the standard monograph \cite{Ki}. Following the line in \cite{S}, in Section 3 of \cite{FS} it was exhibited a Liouville number satisfying the faster approximation {\bf (UL)}.
\vskip.2cm
\smallskip

For a state $\f\in\cs(\ga):=\ga^*_{+,1}$ on a $C^*$-algebra $\ga$, we denote by $(\ch_\f,\pi_\f,\xi_\f)$ its GNS representation, see 
e.g.\  \cite{T1}.
If the cyclic vector $\xi_\f\in \ch_\f$ is also separating for $\pi_\f(\ga)''$, we call $\f$ a \emph{modular state}. We also note that $\f$ being a modular state is equivalent to $s({\f})\in Z(\ga^{**})$, where $s(\f)$ is the support (of the normal extension) of $\f$ in the bidual $\ga^{**}$, see  e.g.\ \cite{NSZ}, pag.\ 15.   

Concerning the Tomita modular theory, we use the notations in \cite{St}. For example, $S_\f$, $J_\f$ and $\D_\f$, with $S_\f=J_\f\D_\f^{1/2}$, denote the Tomita involution, Tomita conjugation and modular operator associated to the modular state $\f\in\cs(\ga)$, respectively.
\smallskip

Let $\idd_\ga\in\gs=\gs^*\subset\ga$ be a, not necessarily closed, self-adjoint subspace of the unital $C^*$-algebra $\ga$ which contains the identity $\idd_\ga$. It is called a (concrete) {\it operator system}.
\medskip

\noindent
\textbf{Noncommutative 2-torus.}
For a fixed $\a\in\br$, the \emph{noncommutative torus} $\ba_{2\a}$ associated to the rotation by the angle $4\pi\a$, is the universal unital $C^*$-algebra with identity $I$ generated by the commutation relations involving two noncommutative unitary indeterminates $U,V$:
\begin{equation}
\label{ccrba}
\begin{split}
&UU^*=U^*U=I=VV^*=V^*V\,,\\
&UV=e^{4\pi\imath\a}VU\,.
\end{split}
\end{equation}
The factor $2$ is considered for the sake of convenience, see below.

We express $\ba_{2\a}$ in the so-called {\it Weyl form}. Indeed, let ${\bf a}:=(m,n) \in \bz^2$  
with $\textbf{0}=(0,0)$, and define
$$
W({\bf a}):=e^{-2\pi\imath\a\, mn}U^mV^n,\quad {\bf a}\in\bz^2\,.
$$
Obviously, $W({\bf 0})=I$, and the commutation relations \eqref{ccrba} become
\begin{equation}
\label{ccrba1}
\begin{split}
&W({\bf a})W({\bf b})=W({\bf a}+{\bf b})e^{2\pi\imath\a\, \s({\bf a},{\bf b})}\,,\\
&W({\bf a})^*=W(-{\bf a}),\quad {\bf a}, {\bf b}\in\bz^2\,,
\end{split}
\end{equation}
where the symplectic form $\s$ is defined by
$$
\s({\bf a},{\bf b}):=(mq-pn),\quad {\bf a}=(m,n), \,{\bf b}=(p,q) \in\bz^2\,.
$$

We now fix a function $f:\bz^2\longrightarrow\bc$, which we may assume to have finite support. The element $W(f)\in\ba_{2\a}$ is then defined as
$$
W(f):=\sum_{{\bf a}\in\bz^2}f({\bf a})W({\bf a})\,.
$$
The set $\{W(f)\in\ba_{2\a}\mid f\in\cb(\bz^2)\,\text{with finite support}\}$ provides a dense $*$-algebra of $\ba_{2\a}$. 

For $\a$ irrational, which is tacitly assumed from now on, we recall that $\ba_{2\a}$ is simple and has a necessarily unique faithful trace $\t$ given by
\begin{equation}
\label{canet}
\t(W(f)):=f({\bf 0}),\quad W(f)\in\ba_{2\a}\,.
\end{equation}

Conversely, by Remark 1.7 of \cite{B}, any element $A\in\ba_{2\a}$ is uniquely determined by the corresponding Fourier coefficients
\begin{equation}
\label{zfua}
f({\bf a}):=\t(W(-{\bf a})A),\quad {\bf a}\in\bz^2\,.
\end{equation}

The relations \eqref{ccrba1} transfer to the generators $W(f)$ as follows:
$$
W(f)^*=W(f^\star),\quad W(f)W(g)=W(f*_{2\a} g)\,,
$$
where
\begin{align*}
&f^\star({\bf a}):=\overline{f(-{\bf a})},\\
&(f*_{2\a} g)({\bf a}):=\sum_{{\bf b}\in\bz^2}f({\bf b})g({\bf a}-{\bf b})e^{-2\pi\imath\a\s({\bf a},{\bf b})}\,,
\end{align*}
is a twisted convolution.

Suppose that the double sequence $f({\bf a})\equiv f(m,n)$ satisfies \eqref{zfua} for some $A\in\ba_{2\a}$. 
It is then easily seen that, for each fixed $n\in\bz$, $f^{(n)}(m):=f(m,n)$ defines a sequence whose Fourier anti-transform 
\begin{equation}
\label{zfua22}
\widecheck{f^{(n)}}(z):=\sum_{m\in\bz}f(m,n) z^m\,,
\end{equation}
where the convergence is understood in norm in the Ces\'aro sense (e.g.\ \cite{D}),
provides a continuous function $\widecheck{f^{(n)}}\in C(\bt)$. Here, for ``Fourier anti-transform" we simply mean that in the series \eqref{zfua22}, we are using $z^m$ instead of $z^{-m}$ as in the usual definition of the Fourier transform: $\widecheck{f^{(n)}}(z)=\widehat{f^{(n)}}(z^{-1})$.

\medskip

\noindent
\textbf{On topological dynamical systems on the unit circle.}
If $\mathpzc{f}:Y\to Y$ is an invertible map on the point-set $Y$:
\begin{itemize} 
\item $\mathpzc{f}^0:=\id_Y$, and 
the inverse of $\mathpzc{f}$ is denoted by $\mathpzc{f}^{-1}$;
\item for the $n$-times composition, $\mathpzc{f}^n:=\underbrace{\mathpzc{f}\circ\cdots\circ \mathpzc{f}}_{n-\text{times}}$;
\item for the $n$-times composition of the inverse, $\mathpzc{f}^{-n}:=\underbrace{\mathpzc{f}^{-1}\circ\cdots\circ \mathpzc{f}^{-1}}_{n-\text{times}}$.
\end{itemize}
Thus, $\mathpzc{f}^{n}$ is meaningful for any $n\in\bz$ with the above convention. 
\medskip

We now specialise the matter 
to $Y$ being the unit circle. 
Given a $C^1$-diffeomorphism $\mathpzc{f}$ of the unit circle $\bt$, its \emph{growth sequence} $\{\g_\mathpzc{f}(n)\mid n\in\bn\}$ is defined as
$$
\g_\mathpzc{f}(n):=\|\partial\mathpzc{f}^n\|_\infty\vee\|\partial\mathpzc{f}^{-n}\|_\infty,\quad n\in\bn
$$
where, with a slight abuse of notation and $z=e^{\imath\th}$, $\partial:=\frac{\di\,\,}{\di z}=\frac1{\imath e^{\imath\th}}\frac{\di\,\,}{\di\th}$ is the derivative 
w.r.t. the variable $z$.
\smallskip

To simplify, we reduce the situation to $C^\infty$-diffeomorphisms $\mathpzc{f}$, called simply ``diffeomorphisms". The diffeomorphisms considered in the present paper are topologically conjugate to the rotation $R_\a$, that is
\begin{equation}
\label{dej}
\mathpzc{f}=\mathpzc{h}_\mathpzc{f}\circ R_\a\circ \mathpzc{h}_\mathpzc{f}^{-1},
\end{equation}
for a unique homeomorphism $\mathpzc{h}_\mathpzc{f}$ of $\bt$ with $\mathpzc{h}_\mathpzc{f}(1)=1$.

If $\mathpzc{f}$ is $C^1$-conjugate to the rotation by some angle $2\pi\a$, i.e.\ $\mathpzc{h}_\mathpzc{f}\in C^1(\bt;\bt)$, then the sequence $(\g_\mathpzc{f}(n))_n$ is bounded. 
As pointed out before, this happens for $C^\infty$-diffeomorphisms $\mathpzc{f}$ with \emph{rotation number}
$\r(\mathpzc{f})=\a$ (see e.g.\ Section 11 of \cite{KH} for the definition of rotation number)  whenever $\a$ is diophantine. 

Conversely, suppose that $\b\in(0,1)\backslash\bq$ and $\mathpzc{f}$ is an orientation preserving diffeomorphism with $\r(\mathpzc{f})=\b$. In this case, the Denjoy Theorem (e.g.\  \cite{KH}) asserts that there exists a unique homeomorphism $\mathpzc{h}_\mathpzc{f}$ of the unit circle such that $\mathpzc{h}_\mathpzc{f}(1)=1$ satisfying \eqref{dej} for the rotation $R_\b$. Then
\begin{equation*}
\n_\mathpzc{f}:=(\mathpzc{h}_\mathpzc{f})^*\mathpzc{m}=\mathpzc{m}\circ \mathpzc{h}_\mathpzc{f}^{-1},
\end{equation*}
is the unique invariant measure, which is ergodic for the natural action of $\mathpzc{f}$ on $\bt$. For a diophantine number $\b$, $\mathpzc{h}_\mathpzc{f}$ is indeed smooth (cf.\ \cite{Yo}) and thus $\n_\mathpzc{f}\sim\mathpzc{m}$. For a Liouville number $\b$, things are quite different. There are diffeomorphisms as above for which the unique invariant measure $(\mathpzc{h}_\mathpzc{f})^*\mathpzc{m}$ is singular w.r.t.\ the Haar measure $\mathpzc{m}$: $(\mathpzc{h}_\mathpzc{f})^*\mathpzc{m}\perp \mathpzc{m}$. 
For $\b=2\a$, in order to exhibit non-type $\ty{II_1}$, and then type $\ty{III}$ representations of the noncommutative torus $\ba_{2\a}$,
we will actually look at such diffeomorphisms, that is those satisfying the above mentioned singularity condition.
The use of the factor 2 is pure matter of convenience, and will be clarified in the sequel. 

Note that, for the class of diffeomorphisms $\mathpzc{f}$ considered in the present paper, the asymptotic of $\g_\mathpzc{f}(n)$ is at most $o(n^2)$, see \cite{W}. However, for the diffeomorphisms constructed in \cite{Ma}, Proposition 2.1 when $\a$ is a Liouville number, or in \cite{FS}, Proposition 3.1 for ultra-Liouville numbers, we have $\g_\mathpzc{f}(n)=o(n)$ or 
$\g_\mathpzc{f}(n)=o(\ln n)$, respectively.

It is of interest to note that it would be possible to select Liouville numbers with sufficiently fast approximation by rationals, providing diffeomorphisms of the kind considered in the present paper such that its growth sequence enjoys any unbounded prescribed asymptotic slower than $o(n)$.
\medskip

\noindent
\textbf{Modular states and non-type $\ty{II_1}$ representations.}
We consider an orientation preserving diffeomorphism $\mathpzc{f}$ on $C^\infty(\bt)$ with rotation number 
$\r(\mathpzc{f})=2\a$. Then it is conjugate to the rotation $4\pi\a$ through an essentially unique homeomorphism as in 
\eqref{dej}. 
For such diffeomorphism, we consider a uniquely determined measure 
$\m_\mathpzc{f}$, and take advantage of the construction in  \cite{FS}, Section 4, of a state $\om_\m$ associated to any probability Radon measure 
$\m\in C(\bt)^*$.

Indeed, let $\m\in\cs(C(\bt))$ be a probability measure on $\bt$. As shown in \cite{A}, Proposition 2.1,
\begin{equation}
\label{ommu}
\om_\m(W(f)):=\sum_{m\in\bz}\widecheck{\m}(m)f(m,0)
\end{equation}
is well defined, positive and normalised, so it defines a state on $\ba_{2\a}$. 

Notice that, if $\m$ is the Haar-Lebesgue measure, then 
$$
\widecheck{\m}(k)=\oint z^k\frac{\di z}{2\pi\imath z}=\d_{k,0}\,,
$$
and thus $\om_\mathpzc{m}=\t$, i.e.\ the trace,  by \eqref{ommu} and \eqref{canet}.
\begin{rem}
For the states $\om_\m$, we have
\begin{itemize}
\item[(i)] $\om_\m\in\cs(\ba_{2\a})$ is faithful if and only if $\supp(\m)=\bt$ (cf.\ \cite{FS}, Proposition 4.1);
\item[(ii)] if for $n\in\bz$, $\m\circ R_{2\a}^{n}\preceq\m$ (and thus $\m\circ R_{2\a}^{n}\sim\m$), then the support of $\om_\m$ in the bidual is central: 
$s(\om_\m)\in Z(\ba_{2\a}^{**})$ (cf.\ \cite{FS}, Proposition 4.2).
\end{itemize}
\end{rem}

Now we specialise the situation to
\begin{equation} 
\label{11ommu}
\m_\mathpzc{f}:=\big(\mathpzc{h}_\mathpzc{f}^{-1}\big)^*\mathpzc{m}=\mathpzc{m}\circ\mathpzc{h}_\mathpzc{f}\,,
\end{equation}
which is quasi invariant and ergodic for the natural action of the rotation $R_{2\a}$ by the angle $4\pi\a$, and thus $\pi_{\om_{\m_\mathpzc{f}}}(\ba_{2\a})''$ acts in standard form on 
$\ch_{\om_{\m_\mathpzc{f}}}$. 

Denoting by $\mathpzc{R}$ the dual action of the rotations  $R_{2\a}$ on functions, and in particular on $L^\infty(\bt,\m_\mathpzc{f})$, it is possible to see that such an action is also free (e.g.\ \cite{T1}, pag. 363). Let now $\mathpzc{F}$ be the dual action of $\mathpzc{f}$ on $L^\infty(\bt,\mathpzc{m})$, as well as on $C(\bt)$.

Since the dynamical systems $\big(L^\infty(\bt,\m_\mathpzc{f}),\mathpzc{R}\big)$ and $\big(L^\infty(\bt,\mathpzc{m}),\mathpzc{F}\big)$ are conjugate, we get (cf.\ \cite{FS}, Proposition 6.2)
that $\pi_{\om_{\m_\mathpzc{f}}}(\ba_{2\a})''$ is indeed a von Neumann crossed product (e.g.\ \cite{BR}, Section 2.7):
$$
L^\infty(\bt,\m_\mathpzc{f})\rtimes_{\mathpzc{R}}\bz\sim\pi_{\om_{\m_\mathpzc{f}}}(\ba_{2\a})''\sim L^\infty(\bt,\mathpzc{m})\rtimes_{\mathpzc{F}}\bz\,,
$$
where in addition (cf.\ \cite{T1}, Theorem XIII.1.5, and Corollary XIII.1.6) $\pi_{\om_{\m_\mathpzc{f}}}(\ba_{2\a})''$ is a factor containing the image of $L^\infty(\bt,\m_\mathpzc{f})$, or equivalently that of 
$L^\infty(\bt,\mathpzc{m})$, as a maximal abelian subalgebra. Since $\bz$ is amenable, $\pi_{\om_{\m_\mathpzc{f}}}(\ba_{2\a})''$ is also hyperfinite by \cite{K2a}, Theorem 4.4.

Since $\ba_{2\a}$ is a simple $C^*$-algebra, we also have
$$
C(\bt)\rtimes_{\mathpzc{R}}\bz\sim\ba_{2\a}\sim C(\bt)\rtimes_{\mathpzc{F}}\bz\,,
$$
where the involved 
products are $C^*$-crossed products, see  \cite{BR}, Section 2.7.
\smallskip

Concerning the type of the von Neumann factors we are obtaining, trivially $\pi_{\om_{\m_\mathpzc{f}}}(\ba_{2\a})''$ cannot be of type $\ty{I}$ because $\m_\mathpzc{f}$ is nonatomic. It is of type $\ty{II_1}$ if and only if the measure class $[\m_\mathpzc{f}]$ contains a probability measure which is invariant under the action generated by $R_{2\a}$, see Theorem XIII.1.7 in \cite{T1}. In this case, 
$\m_\mathpzc{f}\sim \mathpzc{m}$ and $\pi_{\om_{\m_\mathpzc{f}}}(\ba_{2\a})''\sim\pi_{\t}(\ba_{2\a})''$ by uniqueness, and this always occurs when $\a$ is diophantine. 

In the remaining cases, if $\pi_{\om_{\m_\mathpzc{f}}}(\ba_{2\a})''$ is not of $\ty{II_1}$, its type is determined by the Connes invariant ${\rm S}\big(\pi_{\om_{\m_\mathpzc{f}}}(\ba_{2\a})''\big)$, or equally well by
the Krieger-Araki-Woods {\it ratio-set} (cf.\ \cite{K2}) $r([\m_\mathpzc{f}],R_{2\a})=r([\mathpzc{m}],\mathpzc{f})$ associated to the underlying commutative dynamical systems
$\big(\bt,R_{2\a},\m_\mathpzc{f}\big)\sim\big(\bt,\mathpzc{f},\mathpzc{m}\big)$,
because they coincide:
\[
{\rm S}\big(\pi_{\om_{\m_\mathpzc{f}}}(\ba_{2\a})''\big)=r([\m_\mathpzc{f}],R_{2\a})=r([\mathpzc{m}],\mathpzc{f})\,,
\]
see e.g.\  \cite{RMPR}, Section 2.

We remark that diffeomorphisms of the prescribed ratio-set are constructed in \cite{Ma} through limits of a suitable sequence of diffeomorphisms conjugate to rational rotations which approximate the Liouville number $\a$.

We want also to note that $\pi_{\om_{\m_\mathpzc{f}}}(\ba_{2\a})''$ is of type $\ty{II_\infty}$ if and only if the measure class $[\m_\mathpzc{f}]$ contains an unbounded $\s$-finite measure which is invariant under the rotation $R_{2\a}$, and it is of type $\ty{III}$ if and only if the measure class $[\m_\mathpzc{f}]$ contains no $\s$-finite measure invariant under such a rotation $R_{2\a}$.
\smallskip

In order to take into account the intrinsic differential structure of $\ba_{2\a}$, and therefore to define the deformed Dirac operator by deforming the untwisted Dirac operator associated to such an intrinsic differential structure, we directly work on the space $L^\infty(\bt,\mathpzc{m})$ dealing with the action $\mathpzc{F}$ induced by the diffeomorphism $\mathpzc{f}$ on measurable complex valued functions on the unit circle $\bt$. 

Also the ``square-root" $T$ of 
$\mathpzc{f}$, given by
\begin{equation}
\label{ccsr}
T:=\mathpzc{h}_\mathpzc{f}\circ R_\a\circ\mathpzc{h}_\mathpzc{f}^{-1}\,,
\end{equation}
plays a crucial role. Indeed, $T$ satisfying \eqref {ccsr} is a canonical choice for the square root of $\mathpzc{f}$: $T^2=\mathpzc{f}$.

Concerning the GNS representation of $\om_{\m_\mathpzc{f}}$, whose support in the bidual is central by construction in the situation of the present paper, it was shown in \cite{FS} that
\begin{equation}
\label{ggnnss666}
\begin{split}
\ch_{\om_{\m_\mathpzc{f}}}&=\ell^2\big(\bz;L^2(\bt, \mathpzc{m})\big)\cong \bigoplus_{\bz}L^2(\bt, \mathpzc{m})\,,\\
(\pi_{\om_{\m_\mathpzc{f}}}(W(f))g)_n(z)&=\sum_{l\in\bz}\left(\widecheck{f^{(l)}}\circ \mathpzc{h}^{-1}_{T^2}\circ T^{2n-l}\right)(z)g_{n-l}(z)\,,\\
(\xi_{\om_{\m_\mathpzc{f}}})_n(z)&=\d_{n,0},\quad z\in\bt,\,\,n\in\bz\,.
\end{split}
\end{equation}

Considering $x\in\ch_{\om_{\m_\mathpzc{f}}}$, for $n\in\bz$ and $z\in\bt$, the related modular structure is then given by
\begin{align*}
(S_{\om_{\m_\mathpzc{f}}}x)_n(z)=&\overline{(x_{-n}\circ T^{2n})(z)}=\overline{(x_{-n}\circ\mathpzc{f}^{n})(z)}\,,\\
(\D_{\om_{\m_\mathpzc{f}}} x)_n(z)=&\frac{z(\partial T^{2n})(z)}{T^{2n}(z)}x_{n}(z)=\frac{z(\partial\mathpzc{f}^{n})(z)}{\mathpzc{f}^{n}(z)}x_{n}(z)\,,\\ 
(J_{\om_{\m_\mathpzc{f}}}x)_n(z)=&\bigg[\frac{z(\partial T^{2n})(z)}{T^{2n}(z)}\bigg]^{1/2}\overline{(x_{-n}\circ T^{2n})(z)}\\
=&\bigg[\frac{z(\partial\mathpzc{f}^{n})(z)}{\mathpzc{f}^{n}(z)}\bigg]^{1/2}\overline{(x_{-n}\circ\mathpzc{f}^{n})(z)}
\end{align*}
where, as before, $\partial=\frac{\di\,\,}{\di z}$ stands for the derivative w.r.t.\ the variable $z\in\bt$.
\smallskip

For $z\in\bt$ and $n\in\bz$, we set
\begin{equation}
\label{rndica}
\d_n(z):=\frac{\di\big(\mathpzc{m}\circ \mathpzc{f}^{n}\big)}{\di \mathpzc{m}}(z)\equiv\frac{\di\big(\mathpzc{m}\circ T^{2n}\big)}{\di \mathpzc{m}}(z)
=\frac{z(\partial T^{2n})(z)}{T^{2n}(z)}\,,
\end{equation}
so that the modular operator is given by 
\begin{equation}
\label{012rndica}
\D_{\om_{\m_\mathpzc{f}}}=\bigoplus_{n\in\bz}M_{\d_n}\,,
\end{equation}
with $M_g$ denoting the multiplication operator by the function $g$.

\section{Modular spectral triples}
\label{mstripa}

We provide a slightly modified definition of modular spectral triple w.r.t.\ that appeared in Definition 2.4 of \cite{FS}. It takes into account the possibility to deal with operator systems as explained in \cite{CS}.  To put in evidence only the topological and the main algebraic aspects, we limit the analysis to spectral triples of dimension 2, by noticing that it can be straightforwardly generalised to an arbitrary dimension.

Let $\ga$ be a unital $C^*$-algebra. With $\ga\ni a\mapsto*(a):=a^*\in\ga$ we denote the star-operation.

\begin{defin}
\label{mstgl}
A spectral triple considered in the present paper, named as {\it modular spectral triple} according to the terminology in \cite{CM}, is a triplet $\ct:=(\om,\cl,\ce)$, where with a slight abuse of notation we write $\ce\equiv\ce_\ct$, made of a modular state $\om\in\cs(\ga)$, an 
operator $\cl:\ca\subset\ga\to\ga$ defined on a unital dense $*$-algebra $\ca\subset\ga$, and finally an operator system $\ce_\ct$ acting on $\cb(\ch_\omega)$, satisfying the following properties. 

First we put $\cl^\star:=-*\cl\,*$ and, for $a\in \ga$, we denote by $a_\om\in \ch_\om$ the natural embedding of $a$ in $\ch_\om$, i.e.\ $a_\om=\pi_\om(a)\xi_\om$. 
\begin{itemize}
\item [(i)]  With $\ch_\om\supset \ca_\om\ni a_\om\mapsto L_\om^\#a_\om:=(\cl^\#a)_\om\subset\ch_\om$, 
where $\cl^\#$ and $L_\om^\#$ stand for $\cl,\cl^\star$ and $L_\om,L^\star_\om$, respectively, the {\it deformed Dirac operator}
\begin{equation}
\label{spqu1}
D^{(o)}_\ct:=\begin{pmatrix} 
	 0 &\D_\om^{-1}L_\om\\
	L_\om^\star\D_\om^{-1}& 0\\
     \end{pmatrix} 
\end{equation}
acting on its own domain $\cd_{D^{(o)}_\ct}\subset\ch_\om\bigoplus\ch_\om$ uniquely defines a self-adjoint operator with compact resolvent: $D^{(o)}_\ct$ is densely defined essentially self-adjoint 
with closure $D_\ct$, and $\big((1+D_\ct)^2\big)^{-1/2}$ is compact.
\item[(ii)] There exists a dense $*$-subalgebra $\cb\subset\ga$, and a set ${\bf P}=(P_\iota)_{\iota\in\ci}\subset\cb(\ch_\om)$ of (increasing) self-adjoint projections with
$\sup_\iota P_\iota=:P\leq I_{\ch_\om}$, such that 
\[
\ce_\ct:={\rm span}\big\{P,P_\iota\pi_\om(\cb)P_\iota\mid\iota\in\ci\big\}\subset P\cb(\ch_\om)P\,,
\]
and for each $A\in\ce_\ct$, the deformed commutator
\begin{equation}
\label{defcomm}
d^{(o)}_\ct(A)
:=\imath\begin{pmatrix} 
	0&\D_\om^{-1}[L_\om,A]\\
	[L_\om^\star,A]\D_\om^{-1}& 0\\
     \end{pmatrix}
\end{equation}
uniquely defines a bounded operator, denoted by $d_\ct(A)$: $\overline{\cd_{d^{(o)}_\ct(A)}}=\ch_\om\bigoplus\ch_\om$ and 
$$
\sup\big\{\big\|d^{(o)}_\ct(A)\xi\big\|\mid \xi\in\cd_{d^{(o)}_\ct(A)},\,\,\|\xi\|\leq1\big\}<+\infty\,.
$$
\end{itemize}
\end{defin}
\begin{rem}
\noindent
\begin{itemize}
\item[(a)] The usual case of undeformed spectral triples associated to tracial states leads to $\ci$ being the singleton $\{\iota_o\}$ with $P_{\iota_o}=I_{\ch_\om}$. The same situation emerges when $\om$ is an ``inner" deformation of a trace considered in \cite{CM}.
\item[(b)] The main case described in \cite{CS} corresponds to $\D_\om=I_{\ch_\om}$ occurring when $\om$ is a trace, and $\ci$ being a singleton with corresponding projection $P< I_{\ch_\om}$. 
In order not to loose any information possibly encoded into the spectral triple, it is natural to consider a set of projections such that $\sup_\iota P_\iota=I_{\ch_\om}$. 
In such a situation, we say that $\ce_\ct$ is \emph{essential}.

\item[(c)] We may decide not to include in $\ce_\ct$ the irrelevant subspace $\bc P$, where  $P=\sup_{\iota\in \ci} P_\iota$. In such a situation, $\ce$ would be
$$
\ce_\ct:={\rm span}\big\{P_\iota\pi_\om(\ca)P_\iota\mid\iota\in\ci\big\}\subset P\cb(\ch_\om)P\,.
$$
\item[(d)] The models in \cite{FS}, Sections 8 and 9, provide modular spectral triples $\ct=(\om,\cl,\ce_\ct)$ associated to a nontrivial modular structure, see Sections \ref{modstr} and 
\ref{oeost} below. Therefore, we indeed have nontrivial examples of spectral triples satisfying all properties listed in Definition \ref{mstgl}.
\end{itemize}
\end{rem}

For a spectral triple $\ct$ as above, we can define a norm on $\ce_\ct$ as follows:
\begin{equation}
\label{ono}
\snorm A \snorm_\ct:=\|A\|+\|d_\ct(A)\|\,,\quad A\in\ce_\ct\,.
\end{equation}

Notice that this norm is nothing but an analogue of a $C^1$-norm on a space of $C^1$-functions, and can be considered as a kind of a Lipschitz norm on the involved operator systems. 
Moreover, since $d_\ct(A^*)=d_\ct(A)^*$ for each $A\in\ce_\ct$, the $*$-operation in $\ce_\ct$
 is isometric also w.r.t.\ the above norm.
\begin{prop}
The completion $\overline{\ce_\ct}$ w.r.t.\ the norm \eqref{ono} of $\ce_\ct$, can be viewed as a, non necessarily closed, operator system in $\cb(\ch_\om)$, and $d_\ct$ uniquely extends to the whole $\overline{\ce_\ct}$, providing again bounded operators.
\end{prop}
\begin{proof}
We start by noticing that, since $\snorm A^* \snorm_\ct=\snorm A \snorm_\ct$, $\overline{\ce_\ct}$ is equipped with a $*$-operation in a natural way, and thus $\overline{\ce_\ct}$ is an (abstract at this stage) operator system as well. 

Let $\widetilde{A}\in\overline{\ce_\ct}$, and choose a sequence $(A_n)_n\subset \ce_\ct$ converging to $\widetilde{A}$ in the topology generated by $\snorm\,\,\, \snorm_\ct$. 
It is a Cauchy sequence w.r.t.\ the norm \eqref{ono}, and therefore both sequences $(A_n)_n\subset\cb(\ch_\om)$ and 
$\big(d_\ct(A_n)\big)_n\subset\cb(\ch_\om\bigoplus\ch_\om)$ 
are Cauchy sequences w.r.t.\ the norms of 
$\cb(\ch_\om)$ and $\cb(\ch_\om\bigoplus\ch_\om)$, respectively. 

Let $A$ and $B$ the limits in $\cb(\ch_\om)$ and $\cb(\ch_\om\bigoplus\ch_\om)$ of the sequences $(A_n)_n$ and $\big(d_\ct(A_n)\big)_n$, respectively,
and define $j(\widetilde{A}):=A$. It is straightforward to verify that such a map is well defined, and realises a homeomorphism between $\overline{\ce_\ct}$ and 
$j\big(\overline{\ce_\ct}\big)\subset\cb(\ch_\om)$. In addition, after putting $d_\ct(\widetilde{A}):=B$, we obtain the required extension of $d_\ct$ to the whole $\overline{\ce_\ct}$
by bounded operators.
\end{proof}
\begin{rem}
By identifying $\overline{\ce_\ct}$ with its image $j\big(\overline{\ce_\ct}\big)\subset\cb(\ch_\om)$, we can suppose that $\overline{\ce_\ct}$, as well as the operators 
$\{d(A)\mid A\in\overline{\ce_\ct}\}$, are directly acting on the Hilbert space $\ch_\om$ and $\ch_\om\bigoplus\ch_\om$, respectively. In particular, $\overline{\ce_\ct}$ can be considered as a concrete operator system.
\end{rem}

To take into account the whole modular structure associated to a spectral triple $\ct$, we define
\begin{equation}
\label{onoun}
\L_\ct:=\begin{pmatrix}
0&L_\om\\
	L^\star_\om& 0\\
     \end{pmatrix}\,,
\end{equation}
acting on its own domain, which is a dense subspace of $\cb(\ch_\om)\bigoplus\cb(\ch_\om)$ and
$$     
     S_\ct:=\begin{pmatrix}
0&S_\om\\
	-S_\om& 0\\
     \end{pmatrix}\,,\quad
     J_\ct:=\begin{pmatrix}
0&J_\om\\
	-J_\om& 0\\
     \end{pmatrix}
$$
which, with $\g=\begin{pmatrix}
I&0\\
	0&-I\\
     \end{pmatrix}$, leads to $J_\ct^2=-I$ and $J_\ct\g=-\g J_\ct$, the last being the prescription of a (real) spectral triple of dimension 2.
     
We also note that, with $\ca_\om=\{a_\om\mid a\in\ca\}$, $\ca$ as in Definition \ref{mstgl},
we have $[\L_\ct,S_\ct]=0$ on the dense subspace 
$\ca_\om\bigoplus\ca_\om$. The above properties can be viewed as the generalisation to the modular case of the tracial one in Definition 3 of \cite{C}, or in Definition 2.1 of \cite{M}. 

We point out that it is not still completely clear the role of the ``charge conjugation" $J_\om$ and the doubled one $J_\ct$ in the context of modular spectral triples.

\section{Modular spectral triples for the noncommutative 2-torus}
\label{modstr} 

The present section is devoted to check in a detailed way that the models described in Section 8 of \cite{FS} provide modular spectral triples which are mainly associated to non-type $\ty{II_1}$ representations as explained above. 
\smallskip

Indeed, we start with a $C^\infty$-diffeomorphism $\mathpzc{f}$ of the unit circle which is conjugate to the rotation $R_{2\a}$ of the angle $4\pi\a$ as in  \eqref{dej} by the unique homeomorphism $\mathpzc{h}_\mathpzc{f}$ satisfying $\mathpzc{h}_\mathpzc{f}(1)=1$.  
To simplify the notations, we put $\om:=\om_{\m_\mathpzc{f}}$ and remark that, according to \eqref{ggnnss666}, $\ch_{\om_{\m_\mathpzc{f}}}$ does not depend on the normal faithful state $\om$, up to unitary equivalence. In particular, it always coincides with $\ch_\t$ associated to the canonical trace $\t$. We then put 
$$
\ch:=\ch_\om=\ch_\t=\bigoplus_\bz L^2(\bt,\mathpzc{m})
$$ 
for all states $\om$ considered in the present paper.

\vskip.3cm

On the Hilbert space
\begin{equation*}
\begin{split}
\ch\bigoplus\ch\equiv&
\bigg(\bigoplus_\bz L^2(\bt, \mathpzc{m})\bigg)\bigoplus\bigg(\bigoplus_\bz L^2(\bt,\mathpzc{m})\bigg) \\ 
=&\bigoplus_\bz\bigg(L^2(\bt,\mathpzc{m})\bigoplus L^2(\bt,\mathpzc{m})\bigg)\,,
\end{split}
\end{equation*}
we put
$$
D^{(\rm und)}:=\begin{pmatrix} 
	 0 &L \\
	L^\star & 0\\
     \end{pmatrix},
$$
the \emph{undeformed Dirac operator},
where $L$ and $L^\star$ 
are the operators acting on the $L^2$-spaces
coming from the corresponding ones $\cl$ and $\cl^\star$ acting on a norm-dense $*$-algebra $\ca\subset\ba_{2\a}$, as explained in (i) of Definition \ref{mstgl}. Indeed,
\begin{equation}
\label{eldcz}
\cl=\partial_1+\imath\partial_2\,,\quad \cl^\star=-\partial_1+\imath\partial_2\,,
\end{equation}
where $\partial_i:=\frac{\partial\,}{\partial\th_i}$, $i=1,2$, are the partial derivatives of functions $g\big(e^{\imath\th_1},e^{\imath\th_2}\big)$ on the 2-torus $\bt^2$ w.r.t. the angles, and for $\ca$ one can make the choice
\begin{equation}
\label{1accam0}
\ca:=\{W(f)\in\ba_{2\a}\mid f\in\cb(\bz^2)\,\text{with finite support}\}\,.
\end{equation}
We also recall that another reasonable choice would be the dense $*$-subalgebra
$
\ca:=\{W(f)\in\ba_{2\a}\mid f\in\cs(\bz^2)\}\,,
$
$\cs$ standing for the rapidly decreasing functions.
\smallskip

After passing to the Fourier anti-transform \eqref{zfua22} w.r.t.\ the $m$-index, \eqref{eldcz} assumes the form
\begin{equation}
\label{eldcz11}
(\partial_1g)_n(z):=\imath z\frac{\di g_n}{\di z}(z),\quad (\partial_2g)_n(z):=\imath ng_n(z),\quad z\in\bt,\,\,n\in\bz.
\end{equation}

After some easy computions, we get
\begin{equation}
\label{unde}
D^{(\rm und)}=\bigoplus_{n\in\bz}D^{(\rm und)}_n=\bigoplus_{n\in\bz}
     \begin{pmatrix} 
	 0 &\big(\imath z\frac{\di\,\,}{\di z}-n I\big)\\
	\left(-\imath z\frac{\di\,\,}{\di z}-n I\right)& 0\\
     \end{pmatrix},
\end{equation}
where $D^{(\rm und)}_n= \begin{pmatrix} 
	 0 &L_n\\
	L_n^\star& 0\\
     \end{pmatrix}$, with 
\begin{equation}
\label{ellcl}
L_n:=\imath z\frac{\di\,\,}{\di z}-n I\,,\,\,L^\star_n:=-\imath z\frac{\di\,\,}{\di z}-n I\,,
\end{equation} 
both acting on their own domains as dense subspaces of $\ch$. 

For the convenience of the reader, we report the spectral resolution of $D:=\overline{D^{(\rm und)}}$, obtaining
$\s(D)=\left\{\pm\sqrt{m^2+n^2}\mid m,n\in\bn\right\}$.
Concerning the eigenspaces, we put $\eps_{00}^{(\pm)} = \frac1{\sqrt2}\begin{pmatrix} 
 1\\
 \pm1\\
\end{pmatrix}$, and 
$$
\eps_{mn}^{(\pm)}(z) = \frac1{\sqrt2}\begin{pmatrix} 
 \frac{\imath m-n}{\sqrt{n^2+m^2}} \\
 \pm1\\
\end{pmatrix}z^m,\quad (m,n)\in\bz^2\backslash\{(0,0)\},\,z\in\bt\,.
$$
With
$$
e_{nm}^{(\pm)}:=\bigoplus_{k\in\bz}\delta_{n,k}\, \eps_{mn}^{(\pm)} \in \ch\bigoplus\ch,
$$
we see that the orthonormal system $\big\{e_{nm}^{(\pm)}\mid m,n\in\bz\big\}$ is a basis for $\ch\bigoplus\ch$ made of eigenvectors of the untwisted Dirac operator $D$ corresponding to the eigenvalues $\pm\sqrt{m^2+n^2}$. Therefore, each of such eigenspaces is generated by the single eigenvector $e_{nm}^{(\pm)}$ whenever $(m,n)\neq(0,0)$. 
The eigenspace corresponding to 0 (i.e.\ the kernel of D) is degenerate, and is generated by the two eigenvectors $e_{00}^{(\pm)}$.
\medskip

The deformed Dirac operator with the twisting determined by the modular operator, is here defined as follows. 
Pointing out the state $\om$ fixed in the sequel, we put
$\G_\om:=\begin{pmatrix} 
	 \D_\om &0\\
	0& I\\
     \end{pmatrix}$,
and define $D^{(o)}_\om:=\G_\om^{-1}D^{(\rm und)}\G_\om^{-1}$. 

The operator $D^{(o)}_\om$ is essentially self-adjoint in its own domain, whose closure $D_\om$ is easily described as follows.        
Denote with $AC(\bt)$ the set of all absolutely continuous complex valued functions on the unit circle.
The Sobolev-Hilbert space $H^1(\bt)$ is given by
$$
H^1(\bt):=\big\{f\in AC(\bt)\mid f'\in L^2(\bt,\mathpzc{m})\big\}.
$$ 

For each $n\in\bz$, we put
$$
\cd_{D_{\om,n}}:=H^1(\bt)\bigoplus H^1(\bt)\subset L^2(\bt,\mathpzc{m})\bigoplus L^2(\bt,\mathpzc{m}),
$$
and define the (unbounded) operators
\begin{equation}
\label{0ndsd0}
 D_{\om,n}:=\begin{pmatrix} 
	 0 &M_{\d^{-1}_n}L_n\\
	L_n^\star M_{\d^{-1}_n}& 0\\
     \end{pmatrix}\,,
\end{equation}
where the $L_n$ and $L_n^\star$ are given in \eqref{ellcl}.

The deformed Dirac operator $D_\om$ is then defined as
\begin{equation}
\label{ndsd}
D_\om:=\begin{pmatrix} 
	 0 &\D_\om^{-1}L\\
	L^\star\D_\om^{-1}& 0\\
     \end{pmatrix}=\bigoplus_{n\in\bz}
     \begin{pmatrix} 
	 0 &M_{\d^{-1}_n}L_n\\
	L_n^\star M_{\d^{-1}_n}& 0\\
     \end{pmatrix}=\bigoplus_{n\in\bz}D_{\om,n}\,,
\end{equation}
on the domain 
$$
\cd_{D_\om}:=\bigg\{\xi\in\bigoplus_{n\in\bz}\cd_{D_{\om,n}}\mid\sum_{n\in\bz}\|D_{\om,n}\xi_n\|^2<+\infty\bigg\}\,.
$$

Notice that $D_{\om,n}$ has compact resolvent for each $n\in\bz$, and in addition is invertible for each $n\neq0$.

For the convenience of the reader, we report the main properties of the deformed Dirac operator summarised in the following
\begin{thm}
\label{dsukz}
The operator $D_\om$ with domain $\cd_{D_\om}$ is self-adjoint. In addition, 
it has compact resolvent if and only if 
\begin{equation}
\label{kop}
\lim_{n\to\infty}\big\|(D_{\om,n})^{-1}\big\|_{\cb(L^2(\bt,\mathpzc{m}))}=0\,,
\end{equation}
and if 
\begin{equation}
\label{kop01}
\g_\mathpzc{f}(n)=o(n)\,.
\end{equation}
\end{thm}
\begin{proof}
It is matter of routine to check that $D_\om$ is self-adjoint in $\cd_{D_\om}$, see \cite{FS}, Proposition 7.1. Here, we report the details of the proof of Theorem 7.3 in \cite{FS} concerning compactness.

By disregarding the kernel of the Dirac operator, we can suppose that it is invertible. Therefore, $D_\om$ has compact resolvent if and only if $D^{-1}_\om$ is compact. Then it is enough to consider 
$$
D_\#:=\bigoplus_{n\in\bz\smallsetminus\{0\}}D_{\om,n}\,,
$$
which is invertible with inverse
$$
D^{-1}_\#:=\bigoplus_{n\in\bz\smallsetminus\{0\}}D^{-1}_{\om,n}\,.
$$

We now note that, by the form of the operators $D_{\om,n}$ given in \eqref{0ndsd0},
the $D^{-1}_{\om,n}$ are compact operators.

Suppose now that \eqref{kop} holds true, $D^{-1}_\#$ is manifestly compact being norm limit of the sequence of compact operators 
$\big(\bigoplus_{0<|n|\leq N}D^{-1}_{\om,n}\big)_N$.
Conversely, suppose \eqref{kop} were not true. Then there would exist a sequence $(\xi_k)_{k\in\bn}\subset\ch_\om\bigoplus\ch_\om$ of unit vectors such that for $k,l\in\bn$ $k\neq l$,
$$
(D_\#)^{-1}\xi_k\perp (D_\#)^{-1}\xi_l,\quad \big\|(D_\#)^{-1}\xi_k\big\|\geq\eps>0.
$$
Therefore, the sequence $\big((D_\#)^{-1}\xi_k\big)_{k\in\bn}$ is not relatively compact.

Finally, by \cite{FS}, Lemma 7.2, $\big\|D^{-1}_{\om,n}\big\|\leq\frac{\g_\mathpzc{f}(n)}{|n|}$, $n\neq0$. Thus,
if \eqref{kop01} is satisfied, then $\big\|D^{-1}_{\om,n}\big\|=o(1)$ and the second half follows by the first one.
\end{proof}

We now go to exhibit modular spectral triples, some of them arising from non-type $\ty{II_1}$ representations of the noncommutative torus, provided $\a$ is a Liouville number. For such a purpose, define
$$
\cc_0:=\big\{F\circ\mathpzc{h}_{\mathpzc{f}}\mid F\in C^1(\bt)\big\}\subset C(\bt)\,.
$$

For each finite subset $J\subset\bz$, consider the functions $f_k(m,n):=\widehat{F_k}(m)\d_{n,0}$ with $F_k\in \cc_0$, and $g_k(m,n):=\d_{m,0}\d_{n,k}$. We define 
$\cb_0\subset\cb(\bz^2)$ as the set of all functions
\begin{equation*}
f:=\sum_{k\in J}(F_k*_{2\a}g_k),\quad J\subset\bz\,\,\,\text{finite}\,,
\end{equation*}
and put
$$
\ba_{2\a}^{oo}:=\big\{W(f)\mid f\in\cb_0\big\}\,.
$$

In order to obtain a modular spectral triple as in Definition \ref{mstgl}, we set $\ct_\mathpzc{f}:=\big(\om_{\m_\mathpzc{f}},\cl,\ce_{\ct_\mathpzc{f}}\big)$. Here, $\m_\mathpzc{f}$ is the Radon probability measure on $\bt$ defined in \eqref{11ommu}, and $\om$ in the first part is nothing else that $\om_{\m_\mathpzc{f}}$. It is quasi-invariant w.r.t. the multiple of the rotations by the angle $4\pi\a$, and thus 
$\om_{\m_\mathpzc{f}}\in\cs(\ba_{2\a})$ given by \eqref{ommu} has central support in the bidual. 

The intrinsic derivatives $\cl$ and $\cl^\star$ given by \eqref{eldcz}, assuming in our picture the form \eqref{eldcz11},
act on the dense $*$-algebra
$\ca\subset\ba_{2\a}$ given in \eqref{1accam0} providing the required property $\cl^\star=-*\cl*$. It is then possible to get the deformed Dirac operator $D_{\ct_\mathpzc{f}}:=D_{\om_{\m_\mathpzc{f}}}$ given by \eqref{ndsd}. It has compact resolvent if the growth sequence satisfies $\g_\mathpzc{f}(n)=o(n)$, 
see Theorem \ref{dsukz}. The operator system is finally provided by the self-adjoint subspace of $\cb(\ch_{\om_{\m_\mathpzc{f}}})$ given by
$$
\ce_{\ct_\mathpzc{f}}:={\rm span}\{P_N\pi_\om(\ba_{2\a}^{oo})P_N\mid N\in\bn\}\,,
$$
where $P_N$ is the self-adjoint projection in $\cb(\ch_{\om_{\m_\mathpzc{f}}})$ onto 
$\bigoplus_{|n|\leq N}L^2(\bt,\mathpzc{m})$.
Notice that $\ce_{\ct_\mathpzc{f}}+\bc I$ is an operator system,  and $\sup_NP_N=I$. Therefore, $\ce_{\ct_\mathpzc{f}}$ is essential according to Definition \ref{mstgl}.
\begin{rem} 
By \eqref{012rndica}, any $A\in\ce_{\ct_\mathpzc{f}}$ is an entire element w.r.t.\ the adjoint action ${\rm ad}_{\D_{\om_{\m_\mathpzc{f}}}}^{\imath t}$ on 
$\cb\big(\ch_{\om_{\m_\mathpzc{f}}}\big)$ of the one-parameter group generated by the modular operator, see the proof of Theorem \ref{frtor} below.
\end{rem}
\begin{thm}
\label{sczptr} 
The triple $\ct_\mathpzc{f}=\big(\om_{\m_\mathpzc{f}},\cl,\ce_{\ct_\mathpzc{f}}\big)$ is a modular spectral triple, provided \eqref{kop} is satisfied.
\end{thm}
\begin{proof}
With $\om=\om_{\m_\mathpzc{f}}$, we start by noticing that the deformed Dirac operator $D_\om$ given in \eqref{ndsd} has compact resolvent if and only if \eqref{kop} is satisfied.
We notice that a sufficient condition for having compact resolvent is $\g_\mathpzc{f}(n)=o(n)$.
Then, we have to verify (ii) in Definition \ref{mstgl}.

Indeed, we have only to
check that $\D_\om^{-1}[L,P_N\pi_\om(\ba_{2\a}^{oo})P_N]$ provides bounded operators on the common dense domain $\pi_\om(\ca)\xi_\om$ for $\ca$ given in \eqref{1accam0}, the other part concerning the anti-diagonal term in \eqref{spqu1}
being similar.

Since $P_N\pi_\om(\ca)\xi_\om\subset \pi_\om(\ca)\xi_\om$ and, on such a dense domain $P_N$ commutes with $L$ and $\D_\om$, we have 
$$
\D^{-1}_\om[L,P_N\pi_\om(a)P_N]=P_N\D^{-1}_\om[L,\pi_\om(a)]P_N\,,
$$
and thus the assertion follows by \cite{FS}, Theorem 8.1. We report those computations for the convenience of the reader.

Indeed, for the linear generator $a=x \l^l$, where $x:=W(\widehat{F}(m)\d_{n,0})$ and $\l^l:=W(\d_{m,0}\, \d_{n,l})$
the $l$-step shift on the direct sum $\bigoplus_\bz L^2(\bt,\mathpzc{m})$, we have
\begin{align*}
\Delta^{-1}_\omega[L,P_N\pi_\omega(a)P_N]=&P_N\Delta^{-1}_\omega[L,\pi_\omega(a)]P_N\\
=&P_N\Delta^{-1}_\omega[L,\pi_\omega(x)]\lambda^l P_N+P_N\Delta^{-1}_\omega\pi_\omega(x)[L,\lambda^l]P_N\,.
\end{align*}     

We note that $\pi_\om(x)$ is diagonal with $\pi_\om(x)_n=M_{F\circ \mathpzc{f}^{n}}$. Then, setting $A=\pi_\om(x)$, $[L,A]$ is diagonal as well, providing 
\begin{equation*}
[L,A]_n(z)=\imath z\frac{\di\,}{\di z}F(\mathpzc{f}^{n}(z))=\imath z\partial\mathpzc{f}^{n}(z)(\partial F)(\mathpzc{f}^{n}(z)).
\end{equation*}

By using \eqref{rndica} and \eqref{012rndica}, we get
$$
(\D^{-1}[L,A])_n(z)=\imath\mathpzc{f}^{n}(z)\partial F(\mathpzc{f}^{n}(z)),
$$
and thus
\begin{equation}
\label{pfircl}
\big\|\D^{-1}_\om[L,\pi_\om(x)]\big\|\leq\sup_{n\in\bz}\|\mathpzc{f}^n\partial F(\mathpzc{f}^n)\|_\infty
=\sup_{n\in\bz}\|\partial F\|_\infty=\|\partial F\|_\infty\,.
\end{equation}

On the other hand, by (8.3) in \cite{FS}, we have
$$
\big\|P_N\D^{-1}_\om[L,\pi_\om(\l^l)]\big\|\leq |l|\|P_N\D^{-1}_\om\|=|l|\max_{n\leq N}\g_\mathpzc{f}(n)\,,
$$
where $\g_\mathpzc{f}(n)$ is the growth sequence of the diffeomorphism $\mathpzc{f}$.

Collecting together, on the domain $\pi_\om(\ca)\xi_\om$ we get
$$
\big\|\D^{-1}_\om[L,P_N\pi_\om(a)P_N]\|\leq \|\partial F\|_\infty+|l|\max_{n\leq N}\g_\mathpzc{f}(n)<+\infty\,.
$$
\end{proof}
\begin{rem}
For all models considered in the present paper, in order to obtain the estimate \eqref{pfircl}, in the Definition \ref{mstgl}
of the modular spectral triple we should use the modular operator $\D_{\om'_{\m_\mathpzc{f}}}=\D^{-1}_{\om_{\m_\mathpzc{f}}}$ of the state 
$\om'_{\m_\mathpzc{f}}=\langle\,{\bf\cdot}\,\xi_{\om_{\m_\mathpzc{f}}},\xi_{\om_{\m_\mathpzc{f}}}\rangle\lceil_{\pi_{\om_{\m_\mathpzc{f}}}(\ba_{2\a})'}$ on the commutant 
$\pi_{\om_{\m_\mathpzc{f}}}(\ba_{2\a})'$. The meaning of such a fact is still unclear.
\end{rem}

\smallskip

As noticed before, the type of the representation $\pi_{\om_{\m_\mathpzc{f}}}$ to which the modular spectral triple is associated, is uniquely determined by the ratio-set $r([\m_\mathpzc{f}],R_{2\a})=r([\mathpzc{m}],\mathpzc{f})$. For $\a$ being a Liouville number and for the diffeomorphisms $\mathpzc{f}$ considered in Section 2 of  \cite{Ma},  by taking into account \cite{FS}, Proposition 3.1 (vi), it is possible to exhibit modular spectral triples arising from non-type $\ty{II_1}$ representations of the noncommutative torus $\ba_{2\a}$. In such a situation, 
$\ce_{\ct_\mathpzc{f}}$ cannot be chosen as a $*$-algebra, except trivial cases, i.e.\ when the involved states come from inner perturbations of the canonical trace.

\section{Deformed Fredholm modules}
\label{spqu}

Concerning the question of whether a modular spectral triple $\ct=(\om,\cl,\ce_\ct)$ determines some kind of Fredholm module, we could conjecture that such Fredholm modules 
themselves have to be consequently deformed. The reader is referred to \cite{MO} for similar computations relative to the usual (i.e.\ undeformed) Fredholm module.
In order to provide a totally symmetric formula for the quantum differential, we find that such a formula is indeed deformed also when the Dirac operator is not deformed, i.e.\ in the tracial case. 
\smallskip

For the sake of simplicity, we tacitly suppose that $D_\ct$ is invertible.
Since $\ker_{D_\ct}$ is finite dimensional, the assumption of invertibility does not affect the substance of the analysis in the present paper, even if the kernel of the Dirac operator encodes some informations which are relevant for other aspects of noncommutative geometry (e.g.\ \cite{C0}).
Therefore, we can reduce the matter to the subspace of finite codimension $P^\perp_{\ker_{D_\ct}}\ch_\om$, on which $D_\ct P^\perp_{\ker_{D_\ct}}$ is invertible.

In such a situation, the polar decomposition for $D_\ct$ leads to the phase (or, more precisely, the sign)  $F_\ct$ 
and the modulus $|D_\ct|$ satisfying
\begin{equation}
\label{phamod}
\begin{split}
 F_\ct|D_\ct|&= D_\ct=|D_\ct|F_\ct,\\
 D_\ct|D_\ct|^{-1}&=F_\ct\ \supset|D_\ct|^{-1}D_\ct.
\end{split}
\end{equation}
Here, the inclusion $F_\ct\ \supset|D_\ct|^{-1}D_\ct$ means that 
$|D_\ct|^{-1}D_\ct$ is closeble on its natural domain $\cd_{D_\ct}$ and $\overline{|D_\ct|^{-1}D_\ct}=F_\ct$.

In order to deal with the (deformed) Fredholm module, we set
$$
\G_\ct:=\begin{pmatrix} 
	 \D_\om &0\\
	0& I\\
     \end{pmatrix},\,\,
     \L_\ct:=\begin{pmatrix}
0&L_\om\\
	L^\star_\om& 0\\
     \end{pmatrix},\,\,
\underline{A}=\begin{pmatrix} 
	A&0\\
	 0&A\\
     \end{pmatrix},\,\, A\in\ce_\ct\,,
$$
as operators acting on $\ch_\om\bigoplus\ch_\om$,  see also \eqref{onoun}.

To avoid technicalities, conceptually inessential at this stage, we start by supposing that,
for each $A\in \ce_\ct$, 
\begin{equation}
\label{trz}
\!\!\!\!\!\!\!\!\!\!\!\text{(iii)}\,\,\G^{-1}_\ct \underline{A}\G_\ct\,\,\text{and}\,\,
\G_\ct \underline{A}\G_\ct^{-1}\,\,\text{are closable with bounded closure}\,.
\end{equation}

Note that, if $\G_\ct^{\pm 1} \underline{A}\G_\ct^{\mp 1}$ is closable with bounded closure
$\overline{\G_\ct^{\pm 1} \underline{A}\G_\ct^{\mp 1}}$, then $ \underline{A}$ is in the domain of the 
analytic continuation of $\br\ni t\mapsto {\rm ad}_{\G_\ct^{\imath t}}\in{\rm aut}\big(\cb(\ch_\om\bigoplus\ch_\om)\big)$ at $t=\mp\imath$.

We then define in its own domain,
\begin{equation}
\begin{split}
\label{formcomm}
[F_\ct,A]&^{(o)}_\ct :=\imath \big(F_\ct\, \overline{\G_\ct \underline{A}\G_\ct^{-1}}-\overline{\G_\ct^{-1}\underline{A}\G_\ct}\, F_\ct\lceil_{\cd_{D_\ct}}\big)\\
+&\imath\big(F_\ct|D_\ct|\, \overline{\G_\ct \underline{A}\G_\ct^{-1}}\, |D_\ct|^{-1}-|D_\ct|^{-1}\,\overline{\G_\ct^{-1}\underline{A}\G_\ct}\,|D_\ct|F_\ct\lceil_{\cd_{D_\ct}}\big)\,. 
\end{split}
\end{equation}
\begin{defin}
\label{mstglFM1}
The modular spectral triple $\ct$ in Definition \ref{mstgl} is called a \emph{deformed Fredholm module} if in addition, for any $A\in \ce_\ct$, the deformed commutator $[F_\ct,A]^{(o)}_\ct$ in 
\eqref{formcomm} uniquely defines  a compact operator denoted by $[F_\ct,A]_\ct$: $\overline{\cd_{[F_\ct,A]^{(o)}_\ct}}=\ch_\om\bigoplus\ch_\om$, 
$$
\sup\big\{\big\|[F_\ct,A]^{(o)}_\ct\xi\big\|\mid \xi\in\cd_{[F_\ct,A]^{(o)}_\ct},\,\,\|\xi\|\leq1\big\}<+\infty\,,
$$
and $\overline{[F_\ct,A]^{(o)}_\ct}=:[F_\ct,A]_\ct\in\ck(\ch_\om\bigoplus\ch_\om)$.
\end{defin}
\begin{thm}
\label{mstglFM2}
Suppose that the modular spectral triple $\ct$ considered in Definition \ref{mstgl}, in addition to $\rm{(iii)}$ in \eqref{trz}, also satisfies the following conditions:
\begin{itemize}
\item[(iv)] $\underline{A}\car_{\G_\ct^{-1}}\subset \cd_{\G_\ct}$, $A\in\ce_\ct$;
\item[(v)] for $\iota\in \ci$, $P_{\iota}\bigoplus P_{\iota}$ strongly commutes with $\G_\ct$, i.e.\ 
$\G_\ct\big(P_{\iota}\bigoplus P_{\iota}\big)\supset\big(P_{\iota}\bigoplus P_{\iota}\big)\G_\ct$, and
$\big(P_{\iota}\bigoplus P_{\iota}\big)\car_{\L_\ct\G_\ct^{-1}}\subset\cd_{\G_\ct^{-1}}$;  
\item[(vi)] $\overline{\G_\ct \underline{A}\G^{-1}_\ct}\cd_{D_\ct}\subset\cd_{D_\ct}$, $A\in\ce_\ct$.
\end{itemize}

Then, for $\ct$,
\begin{equation}
\label{FM commut}
[F_\ct,A]_\ct=|D_\ct|^{-1}d_\ct(A)+d_\ct(A)|D_\ct|^{-1}\,,\quad A\in\ce_\ct\,,
\end{equation}
and thus it provides a deformed Fredholm module.
\end{thm}
\begin{proof}
By using (iii), (vi) and the polar decomposition \eqref{phamod} of $F_\ct$, from \eqref{formcomm}  we get on the dense domain $\cd_{D_\ct}$,  
\begin{align}
\label{eqcommo}
-\imath [F_\ct,A]^{(o)}_\ct=&|D_\ct|^{-1}\big(D_\ct\, \overline{\G_\ct \underline{A}\G_\ct^{-1}}\,-\overline{\G_\ct^{-1}\underline{A}\G_\ct}\, D_\ct\big)\\
+&\big(D_\ct\, \overline{\G_\ct \underline{A}\G_\ct^{-1}}\,-\,\overline{\G_\ct^{-1}\underline{A}\G_\ct}\, D_\ct\big)|D_\ct|^{-1}. \nonumber
\end{align}

We now check that  $D_\ct\,\overline{\G_\ct \underline{A} \G_\ct^{-1}}- \overline{\G_\ct^{-1} \underline{A} \G_\ct}\,D_\ct\supset d^{(o)}_\ct(A)$. Indeed, by recalling \eqref{spqu1} and using \eqref{defcomm} from Definition \ref{mstgl}, by (iv) and (v) we have
\begin{align*}
-\imath d^{(o)}_\ct(A)=&\G_\ct^{-1}\L_\ct \underline{A}\G_\ct^{-1}-\G_\ct^{-1}\underline{A}\L_\ct \G_\ct^{-1}\\
=&\G_\ct^{-1}\L_\ct \underline{A}\G_\ct^{-1}-\G_\ct^{-1}\underline{A}\big(P_{\iota}\bigoplus P_{\iota}\big)\L_\ct \G_\ct^{-1}\\
=&\G_\ct^{-1}\L_\ct \underline{A}\G_\ct^{-1}-\G_\ct^{-1}\underline{A}\G_\ct\G_\ct^{-1}\big(P_{\iota}\bigoplus P_{\iota}\big)\L_\ct \G_\ct^{-1}\\
=&\G_\ct^{-1}\L_\ct \underline{A}\G_\ct^{-1}-\G_\ct^{-1}\underline{A}\big(P_{\iota}\bigoplus P_{\iota}\big)\G_\ct\G_\ct^{-1}\L_\ct \G_\ct^{-1}\\
=&\G_\ct^{-1}\L_\ct \G_\ct^{-1}\G_\ct\underline{A}\G_\ct^{-1}-\G_\ct^{-1}\underline{A}\G_\ct\G_\ct^{-1}\L_\ct \G_\ct^{-1}\\
=&\left(\G_\ct^{-1}\L_\ct \G_\ct^{-1}\right)\left(\G_\ct\underline{A}\G_\ct^{-1}\right)
-\left(\G_\ct^{-1} \underline{A}\G_\ct\right)\left(\G_\ct^{-1}\L_\ct \G_\ct^{-1}\right)\\
=&D_\ct^{(o)}\left(\G_\ct\underline{A}\G_\ct^{-1}\right)-\left(\G_\ct^{-1} \underline{A}\G_\ct\right)D_\ct^{(o)}\\
\subset&D_\ct\,\overline{\G_\ct \underline{A} \G_\ct^{-1}}- \overline{\G_\ct^{-1} \underline{A} \G_\ct}\,D_\ct\,.   
\end{align*}
Therefore, we find from \eqref{eqcommo},
\begin{equation}
\label{domini}
[F_\ct,A]^{(o)}_\ct\supset|D_\ct|^{-1}d^{(o)}_\ct(A)+d^{(o)}_\ct(A)|D_\ct|^{-1}\,.
\end{equation}
\medskip

We now suppose that $\xi\in \cd_{d^{(o)}_\ct(A)}$. Firstly, $\xi\in \cd_{\G_\ct \underline{A} \G_\ct^{-1}}$,
\[
\overline{\G_\ct \underline{A} \G_\ct^{-1}}\xi= \G_\ct \underline{A} \G_\ct^{-1} \xi \in \cd_{D_\ct^{(o)}}\subset \cd_{D_\ct} \ ,
\]
and 
\[
D_\ct^{(o)}\, \G_\ct \underline{A}\G_\ct^{-1} \xi = D_\ct\, \overline{\G_\ct \underline{A}\G_\ct^{-1}}\xi \ .
\]
Secondly, $\xi\in \cd_{D_\ct^{(o)}}$,  
\[
  D_\ct\xi=D_\ct^{(o)}\xi\in \cd_{\G_\ct^{-1}\underline{A}\G_\ct}\subset\cd_{\overline{\G_\ct \underline{A} \G_\ct^{-1}}} = \ch_\om \bigoplus \ch_\om \  ,
\]
and 
\[
 \G_\ct^{-1} \underline{A}\G_\ct\, D_\ct^{(o)} \xi =  \overline{\G_\ct^{-1} \underline{A}\G_\ct}\, D_\ct \xi \ .
\]

By (ii) in Definition \ref{mstgl} and the previous computations,  
$
\overline{|D_\ct|^{-1} d_\ct^{(o)}(A)}\subset |D_\ct|^{-1} d_\ct(A)
$ and $|D_\ct|^{-1}d_\ct^{(o)}(A)$ is bounded. 
Thus 
$
\overline{|D_\ct|^{-1} d_\ct^{(o)}(A)}=|D_\ct|^{-1} d_\ct(A).
$
At the same time, we also get  
$
\overline{d_\ct^{(o)}(A)|D_\ct|^{-1}}= d_\ct(A)|D_\ct|^{-1}.
$

Collecting together the last computations and \eqref{domini}, 
\begin{align*}
[F_\ct,A]_\ct:=&\overline{[F_\ct,A]^{(o)}_\ct}\supset\overline{|D_\ct|^{-1}d_\ct^{(o)}(A)}+\overline{d_\ct^{(o)}(A)|D_\ct|^{-1}}\\
=&|D_\ct|^{-1} d_\ct(A)+d_\ct(A)|D_\ct|^{-1}\,,
\end{align*}
and thus $[F_\ct,A]_\ct=|D_\ct|^{-1} d_\ct(A)+d_\ct(A)|D_\ct|^{-1}$ because the r.h.s.\ is bounded. Since the r.h.s.\ is also compact, $[F_\ct,A]_\ct$
provides a compact operator for each $A\in\ce_\ct$.
\end{proof}
We note that, when $\om=\t$ is a trace and then $\D_\t=I$, (iv) and (v) in Theorem \ref{mstglFM2} are automatically satisfied, whereas (vi) reduces to
\begin{itemize}
\item[(vi')] $A\cd_{D_\ct}\subset\cd_{D_\ct}$, 
\end{itemize}
which coincides with (1) in Definition 1.1 of \cite{FMR}.
\begin{rem}
The deformed commutator $[F_\ct,A]_\ct$ is also meaningful for $A\in\overline{\ce_\ct}$, providing compact operators acting on $\ch_\om\bigoplus\ch_\om$ as well, again obtaining
$$
[F_\ct,A]_\ct=|D_\ct|^{-1} d_\ct(A)+d_\ct(A)|D_\ct|^{-1}\,,\quad A\in\overline{\ce_\ct}\,.
$$
\end{rem}
\begin{proof}
With $\|\,\,\,\|=\|\,\,\,\|_{\cb(\ch_\om\bigoplus\ch_\om)}$, we consider $\widetilde{A}\in\ce_\ct$, together with a sequence $(A_n)_n\subset\ce_\ct$ converging to $\widetilde{A}$ in the norm $\snorm\,\,\,\snorm_\ct$. 
Since 
\begin{align*}
\big\|[F_\ct,A_n]_\ct-[F_\ct,A_m]_\ct\big\|
&=\big\|[F_\ct,(A_n-A_m)]_\ct\big\|
\leq 2\|D_\ct^{-1}\|\|d_\ct(A_n-A_m)\|\\
&\leq2\|D_\ct^{-1}\|\snorm A_n-A_m\snorm_\ct\to0\,,
\end{align*}
the sequence $\big([F_\ct,A_n]_\ct\big)_n\subset\ck(\ch_\om\bigoplus\ch_\om)$ is a Cauchy sequence which converges to a compact operator as $\ck(\ch_\om\bigoplus\ch_\om)$ is complete. Since such a limit does not depend on the sequence converging to $\widetilde{A}$ in the norm $\snorm \,\,\,\snorm_\ct$, it uniquely defines a compact operator which we still call 
$[F_\ct,\widetilde{A}]_\ct$.
\end{proof}
We point out that, differently to the usual case, the ``quantum differential" $1/|D_\ct|$ appears symmetrically in the formula defining $[F_\ct,A]_\ct$. 
Also for the undeformed spectral triples, that is when $\om$ is a trace, this leads to the additional deformed term $\big(F_\ct|D_\ct|A|D|_\ct^{-1}-|D_\ct|^{-1}A|D_\ct|F_\ct\big)$ to the usual commutator $[F_\ct,A]$.
\medskip

We now provide nontrivial examples associated to the modular spectral triples considered in the present paper for the noncommutative 2-torus, which are indeed deformed Fredholm modules.

\begin{thm}
\label{frtor}
The spectral triples $\ce_{\ct_\mathpzc{f}}$ given in Theorem \ref{sczptr} are indeed deformed Fredholm modules.
\end{thm}
\begin{proof}
To simplify, we put $\om:=\om_{\m_\mathpzc{f}}$, and note that (iv) and (v)  in Definition \ref{mstglFM1} are easily achieved by \eqref{012rndica} and \eqref{unde}.
\smallskip

We now argue that, by \eqref{012rndica} and \eqref{ndsd},
\begin{align}
&\G_\om\big(P_N\bigoplus P_N\big)=\big(P_N\bigoplus P_N\big)\G_\om\big(P_N\bigoplus P_N\big)=\big(P_N\bigoplus P_N\big)\G_\om\,,\label{line1}\\
&D_\om\big(P_N\bigoplus P_N\big)=\big(P_N\bigoplus P_N\big)D_\om\big(P_N\bigoplus P_N\big)=\big(P_N\bigoplus P_N\big)D_\om\,.\label{line2}
\end{align}
Since $\big\|\G_\om\big(P_N\bigoplus P_N\big)\big\|=\max_{n\leq N}\g_\mathpzc{f}(n)$, by \eqref{line1} we have for $A\in\ce_{\ct_\mathpzc{f}}$
$$
\big\|\G_\om^{\imath z}\underline{A}\G_\om^{-\imath z}\big\|\leq\Big(\max_{n\leq N}\g_\mathpzc{f}(n)\Big)^{2|z|}\|A\|<+\infty\,.
$$
Therefore, $\underline{A}$ is an entire element for the group ${\rm ad}_{\G_\om^{\imath t}}$, and thus (iii) in \eqref{trz} is satisfied.

Concerning (vi), we reason as in the proof of Theorem \ref{sczptr}. 
Indeed, for the linear generator $a=x\l^l$, with $x=W(\widehat{F}(m)\d_{n,0})$ and $\l^l=W(\d_{m,0}\d_{n,l})$ as above, we first note that $P_N$ commutes with $\pi_\om(x)$. Therefore, we reduce the matter when $A$ is either $P_N\pi_\om(x)$ or $P_N\pi_\om(\l)^lP_N$.

By \eqref{line2}, and taking into account that the $\d_n$ and $F$ are $C^1$-function on $\bt$, we easily get (vi) for such a situation. For the case $A=P_N\pi_\om(\l)^lP_N$, we also obtain the assertion
after applying again \eqref{line2}. Indeed, if $\xi_k\in H^1(\bt)\bigoplus H^1(\bt)$ and
$$
\xi:=\bigoplus_{n\in\bz}\d_{n,k}\xi_k\in\cd_{D_\om}\subset \bigoplus_\bz\big(L^2(\bt,\mathpzc{m})\bigoplus (L^2(\bt,\mathpzc{m})\big)\,,
$$
since the $\d_n$ are smooth, $\pi_\om(\l)^l\xi=\bigoplus_{n\in\bz}\d_{n,k-l}\xi_k$ is in the domain of $D_\om$ as well.
\end{proof}

\section{A simple illustrative example}
\label{ex}

We briefly discuss a simple example of one dimensional spectral triple concerning the unit circle $\bt$ for which the involved $C^*$-algebra is $C(\bt)$. This example is borrowed from \cite{FMR}, Section 3.1 where, indeed, it was instead considered the algebra $C([0,1])$.
Clearly, the Dirac operator 
$D:=z\frac{\di\,\,}{\di z}=-\imath\frac{\di\,\,}{\di \th}$ 
is nothing else than the ``square root" of the (opposite of) the Laplacian $-\frac{\di^2\,\,}{\di\th^2}$ on the circle. 
In such a way, we have an essentially different
example of spectral triple from that considered in the previous mentioned paper, for which the deformed Fredholm module provided in the present paper can be explicitly managed, even for a tracial state.
\smallskip

We start with the multiplication operator $a:=M_f$ for the function $f(z)=z^l$ and $\xi(z)=\sum_k \xi_k z^k\in L^2\left(\bt,\frac{\di z}{2\pi\imath z}\right)$, where the sum is over a finite set.
As explained above, for our purposes we can disregard the kernel of $D$. 

Since the modular operator is trivial, we put $D|D|^{-1}=F=F_\ct$ and compute
\begin{align*}
-\imath [F,a]_\ct\xi=&(Fa-aF)\xi+(F|D|a|D|^{-1}-|D|^{-1}a|D|F)\xi\\
=&(Fa-|D|^{-1}aD)\xi+(Da|D|^{-1}-aF)\xi\\
=&\sum_{\{k\in \bz\mid k,k+l\neq0\}}\bigg[\bigg(\frac{k+l}{|k+l|}-\frac{k}{|k+l|}\bigg)
+\bigg(\frac{k+l}{|k|}-\frac{k}{|k|}\bigg)\bigg]\xi_k z^{k+l}\\
=&\ l\sum_{\{k\in \bz\mid k,k+l\neq0\}}\bigg(\frac1{|k+l|}+\frac1{|k|}\bigg)\xi_k z^{k+l}\\
=&\big(|D|^{-1}[D,a]+[D,a]|D|^{-1}\big)\xi\,.
\end{align*}

We remark that we are obtaining the perfect analogous formula corresponding to Equation \eqref{FM commut}, even for this simple example of dimension one.
\vskip.2cm

For this simple example, the arising spectral triple $\ct\equiv(\ca,\cl,\ce_\ct)$ of dimension 1, which is also a Fredholm module by the just seen computation, would be chosen as follows.
 For the dense $*$-algebra $\ca$, we can choose 
$$
\ca:=\overline{{\rm span}\{z^k\mid k\in\bz\}}^{\snorm\,\,\snorm_\ct}=C^1(\bt)=:\ce_\ct\, ,
$$
where $\snorm\,\,\snorm_\ct$ is the norm in \eqref{ono} corresponding to that of the uniform convergence of functions together with their 1st derivative. Correspondingly, 
$\cl :=D\lceil_\ca$. Finally, the state $\om$ is obviously given by the Haar-Lebesgue measure  $\frac{\di z}{2\pi\imath z}$ on $\bt$. 
\smallskip

We would like to point out that, conversely to the simple model just described above, Example 3.1 in \cite{FMR} might not provide a Fredholm module for both definitions, the usual undeformed one, and the deformed one in Definition \ref{mstgl} adapted to spectral triples of dimension 1. 
\smallskip

Indeed, on $C([0,2\pi])$ with one of the usual self-adjoint Dirac operators $D_\chi$, for $|\chi|=1$, given by $-\imath\frac{\di\,\,}{\di \th}$ with 
$$
\cd_{D_\chi}=\{f\in AC([0,2\pi])\mid f(0)=\chi f(2\pi),\,f'\in L^2([0,2\pi])\}
$$
(e.g.\ Example in \cite{RS}, Section VIII.2), any dense $*$-subalgebra of $C([0,2\pi])$ on which the commutator with $D_\chi$ provides bounded operators acting on $L^2([0,2\pi])$, would contain a smooth function $f$ for which $f(0)\neq f(2\pi)$. 

For such a kind of function $a:=f(z)$, the asymptotic of the Fourier coefficients $\{\hat f(k)\}_{k\in\bz}$ is at most $1/|k|$ or worst. Therefore, the deformed commutator $[F,a]_\ct$ is meaningless. On the other hand, $[F,a]$ corresponding to the usual definition of Fredholm module, and again defined on any appropriate domain, although giving rise to bounded operators, cannot always provide compact ones as shown in \cite{FMR}.

However, it is an interesting fact that, on $C(\bt)$ and for $D=z\frac{\di\,\,}{\di z}$, the associated undeformed Fredholm module $[F,M_f]=[D|D|^{-1},M_f]$ for $\{M_f \mid f(z)=\sum_{|l|\leq K} f_l\, z^l,\,\,K\in \bn\}$ produces finite rank operators, hence compact ones.
Indeed, set $a=M_f$ with $f(z)=z^l$, and $\xi=\sum_{|k|\leq K}\xi_k z^k$ with $K\in \bn$.

To simplify, we first suppose $l>1$ and compute
\begin{equation}
\label{hhhhhh}
\begin{split}
[F,a]\xi=&\sum_{\{k\in \bz\mid k,k+l\neq0\}}\big(\sign(k+l)-\sign(k)\big)\xi_k z^{k+l}\\
=& 2\sum_{k= -l+1}^{-1}\xi_k z^{k+l} = 2\sum_{k= 1}^{l-1}\xi_{k-l} z^k = 2 P_{[1,\,l-1]} \l^l \xi\, ,
\end{split}
\end{equation}
where $\l$ is the one-step shift defined as $(\l \xi)(z)= \sum_k \xi_{k-1} z^k$, and $P_L$ is the self-adjoint projection given by $(P_L \xi)(z) = \sum_{k\in L} \xi_k z^k$.

Now, in order to manage the general case $l\in\bz$, define  
\begin{equation*}
\left\{\begin{array}{ll}
                    \!\!\!\!\!\!\! &L_l=[l+1,-1]\subset \bz, \quad l<-1\,, \\[1ex]
               	\!\!\!\!\!\!\!&L_l=\emptyset, \quad l\in\{-1, 0, 1\}\,, \\[1ex]
               \!\!\!\!\!\!\!	&L_l=[1,l-1]\subset \bz, \quad l>1
                    \end{array}
                    \right.
\end{equation*}
and, by reasoning as in \eqref{hhhhhh}, it is straightforward to check that 
$$
[F, a]\xi=2\sign(l)P_{L_l} \l^l \xi\,.
$$

We would like to point out that, for $a=M_f$ with $f(z)=\sum_{|l|\leq K} f_l\, z^l$ as above, by \eqref{hhhhhh} $|D|[F,a]$ uniquely defines a bounded operator satisfying
$$
\||D|[F,a]\|\leq 2 \sum_{|l|\leq K} |l| |f_l|\,.
$$
Therefore, denoting by $B_a$ the bounded closure of $|D|[F,a]$, we re-obtain the well-known Connes' formula 
$[F,a]= |D|^{-1}B_a$.

\section{Other examples of modular spectral triples}
\label{oeost}

In the present section, we briefly describe other examples of modular spectral triples and deformed Fredholm modules for the noncommutative 2-torus $\ba_{2\a}$, leaving technical details to the reader. 
\smallskip

The first case we are interested in is described in Section 9 of \cite{FS}, and provides spectral triples $\ct$ for which $\ce_\ct$ is a dense $*$-algebra.
Indeed, in order to obtain a  $*$-algebra, we need 
to fix any orientation preserving diffeomorphism $\mathpzc{f}$ of the circle having a rotation number $\r(\mathpzc{f})=2\a$, where $\a$ is a ultra-Liouville number, see Section \ref{prrrr}.
As usual, we use the notation $\ce_{\ct_\mathpzc{f}}$ for the involved $*$-algebra. 

As previously seen in \cite{FS}, Section 9, $\cl$, and consequently $\cl^\star$, should be modified as
\begin{equation}
\label{eldczmu}
\cl=\partial_1+\imath\partial^{(\mathpzc{f})}_2\,,\quad \cl^\star=-\partial_1+\imath\partial^{(\mathpzc{f})}_2\,,
\end{equation}
where $\partial^{(\mathpzc{f})}_2$ is the multiplier defined in its own domain by
$$
\big(\partial^{(\mathpzc{f})}_2g\big)(z,w):=\imath\!\left(\widecheck{a_{{\mathpzc{f}}\!,n}\, \widehat{g_n}}\right)\!(z,w)\,,
$$
with
$$
a_{{\mathpzc{f}}\!,n}=\sign(n)\sum_{l=1}^{|n|}\frac1{\g_{\mathpzc{f}}\Big(l-\frac{1-\sign(n)}2\Big)},\quad n\in\bz\,.
$$
Here, the Fourier transform and anti-transform involve only the 2nd variable $w$, and $\big(\g_\mathpzc{f}(n)\big)_n$ is the growth sequence of the diffeomorphism $\mathpzc{f}$.

Notice that, for these cases, the choice of $\cl$ and $\cl^\star$ heavily depend on $\mathpzc{f}$, and thus on the state $\om_{\m_\mathpzc{f}}$ entering in the Definition \ref{mstgl} 
of the modular spectral triple. The reader is indeed referred to \cite{FS}, Proposition 3.1 and Section 9, for the proofs and further details. 
\smallskip

A second set of examples is proposed in Section 7 of \cite{FH} for each $\eta\in[0,1]$. With an abuse of notation, we denote by $\ct_\eta$ the modular spectral triples under consideration. We then have a parametric family of deformed Dirac operators given by
\begin{equation}
\label{spetr1129}
D_{\ct_\eta}=\mathrm{D}_{\om,\eta}:=\begin{pmatrix} 
	 0 &\D_\om^{\eta-1}\mathrm{L}_\om\D_\om^{-\eta}\\
	\D_\om^{-\eta}\mathrm{L}_\om^\star\D_\om^{\eta-1}& 0\\
     \end{pmatrix},
\end{equation}
defined in their own domains. Here, $\mathrm{L}_\om$ and $\mathrm{L}^\star_\om$ are the operators in \eqref{eldcz} or in \eqref{eldczmu} acting on the $L^2$-space on their own domains, and \eqref{eldcz} can be clearly viewed as a particular case of  \eqref{eldczmu}. 

Correspondingly, for the deformed derivation $d_{\ct_\eta}$ on $\ce_{\ct_\eta}$, 
\begin{equation}
\label{spetr112}
d_{\ct_\eta}(A)=\imath
\begin{pmatrix} 
	0&\D_\om^{\eta-1}[\mathrm{L},A]\D_\om^{-\eta}\\
	\D_\om^{-\eta}[\mathrm{L}^\star,A]\D_\om^{\eta-1}& 0\\
     \end{pmatrix}
\end{equation}
would still provide bounded operators.
\smallskip

For all values of $\eta$ and $\ct_\eta$ the relative spectral triples, if we consider the models associated to \eqref{eldcz}, the construction  can be carried out for all Liouville numbers, and the undeformed Dirac operator $\begin{pmatrix} 
	0&\mathrm{L}_\om\\
	\mathrm{L}^\star_\om& 0\\
     \end{pmatrix}\equiv\begin{pmatrix} 
	0&L_\om\\
	L^\star_\om& 0\\     \end{pmatrix}$ does not depend on the diffeomorphism $\mathpzc{f}$. For this situation, the price to pay is that 
the domain $\ce_{\ct_\eta}$ of the deformed commutator $d_{\ct_\eta}$ is merely an (essential) operator system. 

In the second case when we use \eqref{eldczmu} for the undeformed Dirac operator $\begin{pmatrix} 
	0&\mathrm{L}_\om\\
	\mathrm{L}^\star_\om& 0\\
     \end{pmatrix}$, $\ce_{\ct_\eta}$ is a dense $*$-algebra, but our construction can be performed only for ultra-Liouville numbers {\bf (UL)} as those described in Section \ref{prrrr}. In addition, the undeformed Dirac operators do depend on $\mathpzc{f}$ according to \eqref{eldczmu}, that is they are not intrinsic objects associated to the noncommutative manifold under consideration.
\smallskip

After setting 
$$
\G_{\ct_\eta}:=\begin{pmatrix} 
	 \D^{1-\eta}_\om &0\\
	0&  \D^{\eta}_\om\\
     \end{pmatrix},\quad
     \L_{\ct_\eta}:=\begin{pmatrix}
0&\mathrm{L}_\om\\
	\mathrm{L}^\star_\om& 0\\
     \end{pmatrix},\quad    
     \underline{A}=\begin{pmatrix} 
	A&0\\
	 0&A\\
     \end{pmatrix},\,\, A\in\ce_\ct\,,
$$
we find that \eqref{spetr1129} and \eqref{spetr112} are obtained respectively as
$$
D_{\ct_\eta}=\G_{\ct_\eta}^{-1} \L_{\ct_\eta}\G_{\ct_\eta}^{-1},\quad d_{\ct_\eta}(A)=\imath\big(D_{\ct_\eta} \G_{\ct_\eta}\underline{A}\G^{-1}_{\ct_\eta}-\G_{\ct_\eta}^{-1}\underline{A}\G_{\ct_\eta}D_{\ct_\eta}\big)\,.
$$

It is expected that for all cases described in the present section, we would still get deformed Fredholm modules. Indeed, with $F_{\ct_\eta}$ the phase of $D_{\ct_\eta}$ and disregarding its kernel, similarly to \eqref{formcomm} for the closure  
$[F_{\ct_\eta},A]_{\ct_\eta}$ of
\begin{equation*}
\begin{split}
&\imath \big(F_{\ct_\eta} \G_{\ct_\eta}\underline{A}\G_{\ct_\eta}^{-1}-\G_{\ct_\eta}^{-1}\underline{A}\G_{\ct_\eta}F_{\ct_\eta}\big)\\
&\quad+\imath\big(F_{\ct_\eta}|D_{\ct_\eta}|\G_{\ct_\eta} \underline{A}\G_{\ct_\eta}^{-1} |D_{\ct_\eta}|^{-1} -|D_{\ct_\eta}|^{-1}\G_{\ct_\eta}^{-1}\underline{A}\G_{\ct_\eta}|D_{\ct_\eta}|F_{\ct_\eta}\big)\,,
\end{split}
\end{equation*}
we would obtain $[F_{\ct_\eta},A]_{\ct_\eta}=|D_{\ct_\eta}|^{-1}d_{\ct_\eta}(A)+d_{\ct_\eta}(A)|D_{\ct_\eta}|^{-1}$ as in Theorem \ref{mstglFM2}.

\section{The spectrum of the Dirac operator and the Hill equation}
\label{hkil}

We now want to discuss the, in general difficult, problem to diagonalise the deformed Dirac operator $D_\om$ 
for the noncommutative 2-torus 
considered in the present paper and given in \eqref{ndsd}. 

In Section \ref{modstr}, we have already noted 
that the set of eigenvectors of $D_\om$ generates the whole Hilbert space.
This easily follows as the $D_{\om, n}$ have compact resolvent for any $n\in \bz$, and thus $\s(D_{\om, n})$ are discrete sets. Therefore, $\s(D_\om)=\overline{\bigcup_{n\in \bz}\s(D_{\om, n})}$.
Theorem \ref{dsukz} provides conditions under which $D_\om$ has compact resolvent, and thus
$\s(D_\om)=\bigcup_{n\in \bz}\s(D_{\om, n})$ if this is the case.
\medskip

We show that the characterisation of eigenvalues and eigenvectors of $D_\om$ is given in terms of a particular class of eigenvalue Hill equations.
\smallskip

The undeformed Dirac operator $\begin{pmatrix} 
	 0 &\overline{L} \\
	\overline{L^\star} & 0\\
     \end{pmatrix}$ gives rise to the eigenvalues equation $\Big(\frac{\di^2\,}{\di\th^2}+(\l^2-n^2)\Big)H=0$ for the periodic function $H$ and each fixed $n\in\bz$, providing eigenvalues and eigenvectors given at the beginning of Section \ref{modstr}.
\medskip

Setting $d_n(\th):=\d_n(e^{\imath \th})$, where the $\d_n$ are defined in \eqref{rndica}, the result concerning the Hill equation and the diagonalization of $D_\om$, independently on 
the fact that it has compact resolvent, are summarised as follows.

\begin{thm}
The real number $\l$ is an eigenvector of the deformed Dirac operator \eqref{ndsd} if and only 
if there exists $n\in\bz$ such that the homogeneous Hill equation
\begin{equation}
\label{ndsd3}
H''+\big(\l^2d_n^{2}(\th)-n^2\big)H=0
\end{equation}
admits a periodic $C^2$ (and therefore necessarily $C^\infty$) nontrivial solution.
\end{thm}
\begin{proof}
Taking into account \eqref{ndsd}, $\l$ is an eigenvalue of $D_\om$ if, for some $n$, there exist non vanishing periodic functions $f,g$ on the unit circle, necessarily smooth, satisfying
\begin{equation}
\label{ndsd1}
\left\{\begin{array}{ll}
                     \!\!\!\!\!\!\!&\big(\imath z\frac{\di\,\,}{\di z}-n I\big)g=\l\d_n f\,, \\[1ex]

               	\!\!\!\!\!\!\!&\big(-\imath z\frac{\di\,\,}{\di z}-n I\big)\d_n^{-1}f=\l g\,.
                    \end{array}
                    \right.
\end{equation}

By replacing $g$ from the 2nd equation in the 1st of \eqref{ndsd1}, and defining $h:=\d_n^{-1}f$, we obtain
\begin{equation*}
\Big(\Big(z\frac{\di\,\,}{\di z}\Big)^2+n^2\Big)h=\l^2\d_n^{2}h\,.
\end{equation*}

By passing to $H(\th):=h(e^{\imath\th})$ and $d_n(\th)=\d_n(e^{\imath\th})$, it becomes
$$
\Big(-\frac{\di^2\,\,}{\di\th^2}+n^2\Big)H=\l^2d_n^{2}H\,,
$$
which leads to the homogeneous Hill equation
\eqref{ndsd3}.
\medskip

Suppose now that \eqref{ndsd3} admits a non zero smooth periodic solution for some $\l\neq0$ and $n\in\bz$, the case $\l=0$ being trivially satisfied.

By putting $H(\th)d_n(\th)=f(z)$ with $z=e^{\imath\th}$, and
\begin{equation}
\label{hcdo1}
g(z):=\frac1{\l}\big(-\imath z\frac{\di\,\,}{\di z}-n I\big)\d_n^{-1}f(z)\,,
\end{equation}
an easy calculation shows that the vector-valued function 
\begin{equation*}
F(z):=\bigoplus_{k\in\bz}\delta_{n,k} f(z)\bigoplus g(z) \in \ch_\om\bigoplus\ch_\om
\end{equation*}
provides an eigenvector corresponding to the eigenvalue $\l$. 

Indeed, for $\l\neq0$ and $n\in\bz$, \eqref{ndsd1} reads
$\d_n^{-1}L_ng=\l f$ and $L_n^\star(\d_n^{-1} f)=\l g$, where $L_n$ and $L_n^\star$ are the differential operators given in \eqref{ellcl}.

We first note that \eqref{hcdo1} is nothing else than the first condition above. Concerning the second, by using the Hill equation we get
\begin{align*}
\d_n^{-1}L_ng=\frac1{\l\d_n}(L_nL_n^\star)\d_n^{-1}f
=\frac1{\l d_n}\Big(-\frac{\di^2\,\,}{\di\th^2}+n^2\Big)H=\l d_n H=\l f\,.
\end{align*}

Summarising,
\begin{align*}
D_\om F=&\begin{pmatrix} 
	 0 &\D_\om^{-1}L\\
	L^\star\D_\om^{-1}& 0\\
     \end{pmatrix}F=
\left(\bigoplus_{n\in\bz}
     \begin{pmatrix} 
	 0 &M_{\d^{-1}_n}L_n\\
	L_n^\star M_{\d^{-1}_n}& 0\\
     \end{pmatrix}\right)F\\
     =&\begin{pmatrix} 
	 0 &M_{\d^{-1}_n}L_n\\
	L_n^\star M_{\d^{-1}_n}& 0\\
     \end{pmatrix}\begin{pmatrix} 
	 f\\
	g\\
     \end{pmatrix}=\l\begin{pmatrix} 
	 f\\
	g\\\end{pmatrix}=\l F\,,
\end{align*}
and this concludes the proof.
\end{proof}
The eigenvalue problem arising from the deformed Dirac operators
\begin{equation}
\label{ndsd4}
D_{\om,\eta}=\bigoplus_{n\in\bz}\begin{pmatrix} 
	 0 &M_{\d_n^{\eta-1}}L_nM_{\d_n^{-\eta}}\\
	M_{\d_n^{-\eta}}L_n^\star M_{\d_n^{\eta-1}}& 0\\
     \end{pmatrix}\,,
    \end{equation}
defined for $\eta\in[0,1]$ in \eqref{spetr1129}, provides a generalisation that can be achieved as follows. By \eqref{ndsd4}, we get
\begin{equation*}
\left\{\begin{array}{ll}
              &\!\!\!\!\!\!\!\big(\imath z\frac{\di\,\,}{\di z}-n I\big)\d_n^{-\eta}g=\l\d_n^{1-\eta}f\,, \\[1ex]

               	&\!\!\!\!\!\!\!\big(-\imath z\frac{\di\,\,}{\di z}-n I\big)\d_n^{\eta-1}f=\l\d_n^{\eta}g\,,
                    \end{array}
                    \right.
\end{equation*}

By arguing as in the case corresponding to $\eta=0$ previously described, with $H(\th):=d_n(\th)^{\eta-1}f(e^{\imath\th})$, we obtain after some computations,
\begin{equation*}
H''-2\eta(\ln d_n(\th))'H'+\big(\l^2d_n^{2}(\th)-2\eta(\ln d_n(\th))'-n^2\big)H=0\,.
\end{equation*}
\bigskip

\noindent
\textbf{\large Acknowledgment}

\noindent
The authors acknowledge MIUR Excellence Department Project awarded to the
Department of Mathematics, University of Rome ``Tor Vergata'', CUP
E83C18000100006, and Italian INdAM-GNAMPA. 
The first-named author is partially supported by MIUR-FARE R16X5RB55W QUEST-NET.
The second-named author is grateful to the Faculty of Mathematics and Computer Science of University of W\"urzburg for the kind hospitality in occasion of the Prodi-Chair position (winter 2021-2022), where the present work was finalised.

\end{document}